\documentclass{amsart}
\usepackage{amsmath,amssymb,amsbsy,amscd,amsthm,mathrsfs,ulem}

\theoremstyle{plain}
\newtheorem{thm}{Theorem}[section]
\newtheorem{lem}[thm]{Lemma}
\newtheorem{prop}[thm]{Proposition}
\newtheorem{cor}[thm]{Corollary}
\newtheorem{conj}[thm]{Conjecture}
\newtheorem{claim}[thm]{Claim}
\newtheorem{q}[thm]{Question}
\newtheorem{notation}[thm]{Notation}

\theoremstyle{definition}
\newtheorem{definition}{Definition}

\newtheorem{example}{Example}

\theoremstyle{rem}
\newtheorem{rem}{Remark}

\newcommand{\fp }{{\mathfrak{p}}}

\newcommand{\fn }{{\mathfrak{n}}}

\newcommand{\bF }{{\bf F}}

\newcommand{\Ga}{{\bf G}_a}
\newcommand{\A}{{\bf A}}

\newcommand{\Spec }{\mathop{\rm Spec}\nolimits}
\newcommand{\nil }{\mathop{\rm nil}\nolimits}
\newcommand{\pl }{\mathop{\rm pl}\nolimits}

\newcommand{\trd }{\mathop{\rm tr.deg}\nolimits}
\newcommand{\lt }{\mathop{\rm lt}\nolimits}

\newcommand{\zs}{\{ 0\} }
\newcommand{\sm}{\setminus}

\newcommand{\Aut }{\mathop{\rm Aut}\nolimits}
\newcommand{\Aff }{\mathop{\rm Aff}\nolimits}
\newcommand{\T }{\mathop{\rm T}\nolimits}

\newcommand{\El }{\mathop{\rm El}\nolimits}
\newcommand{\Ex }{\mathcal{E}}
\newcommand{\Ch }{\mathcal{C}}
\newcommand{\ch }{\mathop{\mathrm{char}}\nolimits}
\newcommand{\id}{{\rm id}}

\newcommand{\nd}{\noindent}
\newcommand{\ol}{\overline}

\newcommand{\Z}{{\bf Z}}
\newcommand{\ep}{\epsilon}
\newcommand{\x}{{\boldsymbol x}}

\newcommand{\y}{{\boldsymbol y}}
\newcommand{\z}{{\boldsymbol z}}

\newcommand{\bp}{{\boldsymbol p}}

\newcommand{\kx}{k[\x]}
\newcommand{\kxr}{k(\x)}
\newcommand{\ky}{k[\y ]}

\newcommand{\Ry}{R[\y]}
\newcommand{\Sy}{S[\y]}

\begin{document}

\title[Polynomial automorphisms of characteristic order]
{Polynomial automorphisms of characteristic order and their invariant rings}

\author{Shigeru Kuroda}
\address{Department of Mathematical Sciences\\
Tokyo Metropolitan University\\
1-1 Minami-Osawa, Hachioji, Tokyo, 192-0397, Japan}
\email{kuroda@tmu.ac.jp}
\thanks{This work is partly supported by JSPS KAKENHI
Grant Number 18K03219}


\maketitle

\begin{abstract}

Let $k$ be a field of characteristic $p>0$. 
We discuss the automorphisms of the polynomial ring 
$k[x_1,\ldots ,x_n]$ of order $p$, 
or equivalently the $\Z /p\Z $-actions on the affine space $\A _k^n$. 
When $n=2$, 
such an automorphism is know to be a conjugate of an automorphism 
fixing a variable. 
It is an open question whether the same holds when $n\ge 3$.

In this paper, 
(1) we give the first counterexample to this question when $n=3$. 
In fact, 
we show that every $\Ga $-action on $\A _k^3$ of rank three 
yields counterexamples for $n=3$. 
We give a family of counterexamples by constructing 
a family of 
rank three $\Ga $-actions on $\A _k^3$. 
(2) For the automorphisms induced by 
this family of $\Ga $-actions, 
we show that 
the invariant ring is isomorphic to $k[x_1,x_2,x_3]$ 
if and only if the plinth ideal is principal, 
under some mild assumptions. 
(3) We study the Nagata type automorphisms of $R[x_1,x_2]$, 
where $R$ is a UFD of characteristic $p>0$. 
This type of automorphisms are of order $p$. 
We give a necessary and sufficient condition 
for the invariant ring to be isomorphic to $R[x_1,x_2]$. 
This condition is equivalent to the condition that 
the plinth ideal is principal. 

\end{abstract}

\section{Introduction}\label{sect:intro}
\setcounter{equation}{0}

Let $k$ be a field of characteristic $p\ge 0$, 
$\kx =k[x_1,\ldots ,x_n]$ 
the polynomial ring in $n$ variables over $k$, 
and $\Aut _k\kx $ the automorphism group 
of the $k$-algebra $\kx $. 
For $\phi \in \Aut _k\kx $, 
we consider 
the invariant ring $\kx ^{\phi }:=\{ f\in \kx \mid \phi (f)=f\} $. 
We say that $\phi \in \Aut _k\kx $ is 

\nd 
$\bullet $ 
{\it affine} if $\deg \phi (x_i)=1$ 
for $i=1,\ldots ,n$; 

\nd 
$\bullet $ 
{\it elementary} if $x_1,\ldots ,x_{n-1}\in \kx ^{\phi }$ 
and $\phi (x_n)\in x_n+k[x_1,\ldots ,x_{n-1}]$; 

\nd 
$\bullet $ 
{\it exponential} if $\phi $ 
is induced by a $\Ga $-action on the affine space $\A _k^n$ 
(cf.~\S \ref{sect:Ga-action});

\nd 
$\bullet $ 
{\it of characteristic-order} 
if $\langle \phi \rangle \simeq \Z /p\Z $ 
or $\phi =\id $.

Let $\Aff _n(k)$ 
(resp.\ $\El _n(k)$, $\Ex _n(k)$, and $\Ch _n(k)$) 
be the set of $\phi \in \Aut _k\kx $ 
which is affine (resp.\ elementary, exponential, 
and of characteristic-order). 
Then, we have 
$$
\T _n(k):=\langle \Aff _n(k)\cup \El _n(k)\rangle 
\subset \langle \Aff _n(k)\cup \Ex _n(k)\rangle 
\subset \langle \Aff _n(k)\cup \Ch _n(k)\rangle  
\subset \Aut _k\kx , 
$$
since $\El _n(k)\subset \Ex _n(k)\subset \Ch _n(k)$ 
(cf.~\S \ref{sect:Ga-action}). 
We call $\T _n(k)$ the {\it tame subgroup}. 
The {\it Tame Generators Problem} asks whether 
$\T _n(k)=\Aut _k\kx $. 
This is clear if $n=1$. 
Jung~\cite{Jung} and van der Kulk~\cite{Kulk} 
showed that $\T _2(k)=\Aut _kk[x_1,x_2]$. 
In 2004, 
Shestakov-Umirbaev~\cite{SU} 
showed that the automorphism of Nagata~\cite{Nagata} 
does not belong to $\T _3(k)$ if $p=0$, 
and solved the problem in the negative 
when $n=3$ and $p=0$. 
At present, 
the problem is open when $n=3$ and $p>0$, and when $n\ge 4$.

It is well known that Nagata's automorphism is exponential. 
Hence, 
$\T_3(k)$ is a proper subgroup of 
$\langle \Aff _3(k)\cup \Ex _3(k)\rangle $ if $p=0$. 
The {\it Exponential Generators Conjecture} 
asserts that 
$\Aut _k\kx =\langle \Aff _n(k)\cup \Ex _n(k)\rangle $ 
(cf.~\cite[\S 2.1]{Essen}), 
which is open for all $n\ge 3$. 
If $p=0$, 
then we have 
$\Aut _k\kx =\langle \Aff _n(k)\cup \Ch _n(k)\rangle $, 
since $\Ch _n(k)$ contains $\phi \in \Aut _k\kx $ 
whenever the Jacobian of $\phi$ is not a root of unity. 
It is not known if the same holds when $p>0$ and $n\ge 3$.

To study $\Aut _k\kx $ when $p>0$, 
we consider 
$\Ch _n(k)$ to be important. 
The automorphisms of order $p$ 
are equivalent to 
the $\Z /p\Z $-actions, 
and there are many researches on this and related subjects. 
For example, 
Miyanishi~\cite{M1} investigated $\Z /p\Z $-actions 
on a normal affine domain of characteristic $p$ 
from a view point of Artin-Schreier coverings 
(see also Takeda~\cite{Takeda}). 
Miyanishi-Ito~\cite{MI} contains more background in this direction. 
Tanimoto~\cite{Tani} studied $\Z /p\Z $-actions on $\A _k^n$ 
from an interest of Modular Invariant Theory. 
He classified the triangular $\Z /p\Z $-actions on $\A _k^3$ 
and showed that their invariant rings are generated by at most four elements. 
We also mention that Maubach~\cite{Maubach} showed that 
the invariant rings for a certain class of 
$\Z /p^n\Z $-actions on $\A _k^n$ 
are isomorphic to $\kx $. 
In general, 
for an action of 
a finite group $G$ with $p\nmid |G|$ on $\A _k^n$, 
it is difficult to describe the structure of the invariant ring, 
even for a linear action (cf.~\cite{MIT}). 
It should also be stressed that, 
to study properties of $\phi \in \Aut _k\kx $, 
the information about $\kx ^{\phi }$ is of great use.

Recently, 
some researchers remarked that, 
if $p>0$, 
then $\phi \in \Ch _2(k)$ 
is always a conjugate of an elementary automorphism 
(cf.~Theorem~\ref{thm:Osaka}, \cite{M1}, \cite{Maubach}). 
Hence, 
there exists $\sigma \in \Aut _kk[x_1,x_2]$ 
such that $\sigma (x_1)\in k[x_1,x_2]^{\phi }$. 
Then, the following question naturally arises. 
Here, 
for each $k$-subalgebra $A$ of $\kx $, 
we define 
$$
\gamma (A):=\max \{ N\mid 
\exists \sigma \in \Aut _k\kx \text{ such that\ }
\sigma (k[x_1,\ldots ,x_N])\subset A\} . 
$$

\begin{q}\label{question}\rm
Does $\gamma (\kx ^{\phi })\ge 1$ hold for all $\phi \in \Ch _n(k)$ 
when $p>0$ and $n\ge 3$? 
\end{q}

We have three main contributions in this paper.

\nd 1) 
We give the first 
counterexample to Question~\ref{question} for $n=3$. 
To explain the result, 
we recall some known results about $\Ga $-actions on $\A _k^n$. 
The {\it rank} 
of a $\Ga $-action on $\A _k^n$ is defined to be $n-\gamma (\kx ^{\Ga })$ 
(cf.~\cite{Frank}). 
Then, every nontrivial $\Ga $-action on $\A _k^2$ is of rank one 
if $p=0$ by Rentschler~\cite{Rentschler}, 
and if $p>0$ by Miyanishi~\cite{{MiyanishiNagoya}}. 
When $p=0$, 
Freudenburg~\cite{Frank} 
gave the first example of a $\Ga $-action on $\A _k^n$ of rank $n$ 
for each $n\ge 3$. 
When $p>0$, 
our result says that 
every rank three $\Ga $-action on $\A _k^3$ 
yields a family of $\phi \in \Ex _3(k)$ with $\gamma (\kx ^{\phi })=0$ 
(Theorem~\ref{thm:main}). 
Here, 
we emphasize that $\gamma (\kx ^{\Ga })=0$ 
does not immediately imply $\gamma (\kx ^{\phi })=0$ 
for an induced $\phi \in \Ex _3(k)$, 
because $\kx ^{\Ga }\subsetneq \kx ^{\phi }$. 
We construct a family of rank three 
$\Ga $-actions on $\A _k^3$ when $p>0$, 
and give counterexamples to Question~\ref{question}. 
See (\ref{eq:simple example}) for simple concrete examples.

\nd 2) 
The plinth ideal $\pl (\phi )$ for an automorphism $\phi $ 
(cf.~(\ref{eq:def pl phi})) is an analogue of 
the plinth ideal for a derivation (cf.~\cite[\S 1.1]{Plinth}), 
which carries useful information about $\phi $. 
For $\phi \in \Ex _3(k)$ 
induced by the rank three $\Ga $-action on $\A _k^3$ 
stated above, 
we show that $\pl (\phi )$ is principal 
if and only if $\kx ^{\phi }$ is isomorphic to $\kx $ 
under some mild assumptions 
(Theorems~\ref{thm:plinth rank 3}, 
\ref{thm:rank3 invariant ring} 
and \ref{thm:rank3 invariant ring2}). 
This result is of interest in its own right, 
because a $\Ga $-action on $\A _k^n$ of rank $n$ 
is in general difficult and mysterious.

\nd 3) 
Let $R$ be a domain, 
$0\ne a\in R$, $0\ne \theta (x_2)\in x_2R[x_2]$ 
and $0\ne F\in R[ax_1+\theta (x_2)]$. 
Then, 
there exists 
$\psi \in \Aut _RR[x_1,x_2]$ 
such that 
$\psi (x_1)=x_1+a^{-1}(\theta  (x_2)-\theta (x_2+aF))$ 
and $\psi (x_2)=x_2+aF$ 
(cf.~\S \ref{sect:Nagata construction}), 
and is called the {\it Nagata type automorphism}. 
It is known that 
$\psi $ is exponential. 
Hence, 
$\psi ^p=\id $ holds if $\ch R=p>0$. 
Nagata's automorphism is equal to $\psi $ with 
$(R,a,\theta  (x_2),F)=(k[x_3],x_3,x_2^2,x_1x_3+x_2^2)$.

With this notation, 
we have the following result.

\begin{thm}\label{thm:Nagata Main}
Assume that $R$ is a UFD with $\ch R>0$, 
and let $\psi $ be as above. 

\nd{\rm (i)} 
The invariant ring $R[x_1,x_2]^{\psi }$ is generated by 
at most three elements over $R$. 

\nd{\rm (ii)} The following are equivalent: 

{\rm (a)} The ideal $I:=(a,d\theta (x_2)/dx_2)$ 
of $R[x_1,x_2]$ is principal. 

{\rm (b)} The plinth ideal 
$\pl (\psi )$ is a principal ideal of $R[x_1,x_2]^{\psi }$. 

{\rm (c)} $R[x_1,x_2]^{\psi }$ is isomorphic to $R[x_1,x_2]$ 
as an $R$-algebra. 

\nd{\rm (iii)} If $R=k[x_3,\ldots ,x_n]$, 
then {\rm (a)}, {\rm (b)} and {\rm (c)} 
in {\rm (ii)} are equivalent to the following: 

{\rm (d)} $R[x_1,x_2]^{\psi }$ is isomorphic to $R[x_1,x_2]$ 
as a $k$-algebra. 
\end{thm}

This paper is organized as follows. 
In Section~\ref{sect:prelim}, 
we recall basic notions and results used in this paper. 
In Section~\ref{sect:key}, 
we discuss how to derive a counterexample to 
Question~\ref{question} from 
a rank three $\Ga $-action on $\A _k^3$. 
In Sections~\ref{sect:rank3 family} and \ref{sect:invariant ring}, 
we construct a family of rank three $\Ga $-actions on $\A _k^3$, 
and study their exponential automorphisms. 
Section~\ref{sect:Nagata} is devoted to the study of 
the Nagata type automorphisms. 
In Section~\ref{sect:remark}, 
we list some questions and conjectures.

\section{Preliminary}\label{sect:prelim}
\setcounter{equation}{0}

Throughout this paper, 
all rings and algebras are commutative, 
and $k$ denotes a field. 
If $B\subset B'$ are domains, 
$\trd _BB'$ denotes the transcendence degree of $B'$ over $B$, 
and $Q(B)$ denotes the quotient field of $B$. 
For rings $B\subset B'$ 
and a ring homomorphism $\phi :B\to B'$, 
we define 
\begin{equation}\label{eq:B^phi def}
B^{\phi }:=\{ b\in B\mid \phi (b)=b\} .
\end{equation}

\subsection{$\Ga $-action and exponential automorphisms}\label{sect:Ga-action}

Let $R$ be a ring, $B$ an $R$-algebra, 
and $T$ and $U$ indeterminates. 
Recall that a homomorphism $\ep :B\to B[T]$ of $R$-algebras 
defines an action of the additive group 
$\Ga :=\Spec R[T]$ on $\Spec B$ 
if and only if the following conditions 
hold for each $a\in B$. 
Here, we write $\ep (a)=\sum _{i\ge 0}a_iT^i$, 
where $a_i\in B$. 

\smallskip 

(A1) $a_0=a$. \qquad 
(A2) $\sum_{i\ge 0}\ep (a_i)U^i
=\sum _{i\ge 0}a_i(T+U)^i$ in $B[T,U]$. 

\smallskip 

\nd 
If this is the case, 
the $\Ga $-invariant ring is 
$B^{\ep }$. 
We call this $\ep $ a $\Ga $-{\it action on} $B$.

Let $\ep :B\to B[T]$ be a $\Ga $-action on $B$. 
For each $a\in B^{\ep }$, 
we define 
$$
\ep _a:B\ni b\mapsto \ep (b)|_{T=a}\in B, 
$$ 
where $\ep (b)|_{T=a}$ 
is the value of $\ep (b)\in B[T]$ at $T=a$. 
Clearly, 
we have $B^{\ep }\subset B^{\ep _a}$. 
Note that $\ep _0=\id $ by (A1), 
and $\ep _a\circ \ep _b=\ep _{a+b}$ 
for all $a,b\in B^{\ep }$ by (A2). 
Hence, 
$\ep _a$ has the inverse $\ep _{-a}$, 
and $B^{\ep }\ni a\mapsto \ep _a\in \Aut _RB$ is 
a group homomorphism.

\begin{definition}\label{def:exp}\rm 
We say that $\phi \in \Aut _RB$ is {\it exponential} 
if $\phi =\ep _a$ for some $\Ga $-action $\ep $ on $B$ and 
$a\in B^{\ep }$. 
If this is the case, 
we have $B^{\ep }\subset B^{\phi }$. 
If moreover $\ch R=p>0$, 
then we have $\phi ^p=\id $, 
since $(\ep _a)^p=\ep _{pa}=\ep _0=\id $. 
\end{definition}

\begin{example}\label{example:action on A[x]}\rm 
Let $A$ be an $R$-domain, 
and $A[x]$ the polynomial ring in one variable over $A$. 
For $a\in A\sm \zs $, 
we define 
$\widetilde{\ep }:A[x]\ni f(x)\mapsto f(x+aT)\in A[x][T]$. 

\nd (i) 
$\widetilde{\ep }$ 
is a $\Ga $-action on $A[x]$ 
with $A[x]^{\widetilde{\ep }}=A$. 

\nd (ii) 
$\widetilde{\ep }_b$ is equal to $A[x]\ni f(x)\mapsto f(x+ab)\in A[x]$ 
for each $b\in A$.

\nd (iii) 
Let $S$ be an $R$-subalgebra of $A[x]$ 
with $\widetilde{\ep }(S)\subset S[T]$. 
Then, $\widetilde{\ep }$ restricts to a $\Ga $-action $\ep $ on $S$ 
with $S^{\ep }=A\cap S$. 
In this case, 
$\ep _b$ is the restriction of $\widetilde{\ep }_b$ to $S$ 
for each $b\in A\cap S$. 
\end{example}

The elementary automorphisms of $\kx $ are the exponential automorphisms 
for the $\Ga $-actions as in Example~\ref{example:action on A[x]} 
with $A=k[x_1,\ldots ,x_{n-1}]$ and $x=x_n$.

Finally, we recall the following well-known fact (cf.~\cite{M1}).

\begin{rem}\label{rem:Miyanishi}\rm 
Let $B$ be a $k$-domain, 
and $\ep $ a $\Ga $-action on $B$. 

\nd (i) $B^{\ep }$ is {\it factorially closed} in $B$, 
i.e., $ab\in B^{\ep }$ implies $a,b\in B^{\ep }$ 
for each $a,b\in B\sm \zs $. 

\nd (ii) $B^{\ep }$ is {\it algebraically closed} in $B$, 
i.e., $a\in B$ belongs to $B^{\ep }$ 
if $f(a)=0$ for some $f(T)\in B^{\ep }[T]\sm B^{\ep }$.

\nd (iii) If $B^{\ep }\ne B$ and $\trd _kB<\infty $, 
then we have $\trd _kB^{\ep }=\trd _kB-1$. 
\end{rem}

\subsection{One and two variable cases}

Let $p$ be a prime number, 
$R$ a ring with $\ch R=p$, 
and $R[x]$ the polynomial ring in one variable over $R$. 
For $a\in R\sm \zs $, 
we define 
$\phi \in \Aut _RR[x]$ by $\phi (x)=x+a$. 
Then, $\phi $ is of order $p$, 
since $\bF _p\subset R$. 
Moreover, we have $x^p-a^{p-1}x\in R[x]^{\phi }$, since 
\begin{equation}\label{eq:x^p-ax}
\phi (x^p-a^{p-1}x)
=(x^p+a^p)-a^{p-1}(x+a)
=x^p-a^{p-1}x. 
\end{equation}
With this notation, 
the following lemma holds.

\begin{lem}\label{lem:1var}
If $a$ is not a zero-divisor of $R$, 
then we have 
$R[x]^{\phi }=R[x^p-a^{p-1}x]$. 
\end{lem}
\begin{proof}
We show that 
$f$ belongs to $R[x^p-a^{p-1}x]$ 
for all $f\in R[x]^{\phi }$ 
by induction on $l:=\deg f$. 
The assertion is clear if $l\leq 0$. 
Assume that $l\geq 1$, 
and let $b\in R\sm \zs $ be the leading coefficient of $f$. 
Then, 
we have $0=\phi (f)-f=labx^{l-1}+\cdots $. 
Since $a$ is not a zero-divisor of $R$, 
this implies $p\mid l$. 
Set $f':=f-b(x^p-a^{p-1}x)^{l/p}\in R[x]^{\phi }$. 
Then, 
$\deg f'$ is less than $l$. 
Hence, 
$f'$ belongs to $R[x^p-a^{p-1}x]$ by induction assumption. 
Therefore, 
$f=f'+b(x^p-a^{p-1}x)^{l/p}$ belongs to 
$R[x^p-a^{p-1}x]$. 
\end{proof}

The following theorem\footnote{
The author announced this theorem, 
together with a counterexample to Question~\ref{question}, 
on the occasion of the 13th meeting 
of Affine Algebraic Geometry at Osaka 
on March 5, 2015 (see \cite{Miyanishi3}).} 
(cf.~\cite{Miyanishi3}, \cite{Maubach}) 
is based on 
the well-known fact that $\Aut _kk[x_1,x_2]$ is the amalgamated 
product of $\Aff _2(k)$ and the triangular subgroup.

\begin{thm}\label{thm:Osaka}
Let $k$ be a field of characteristic $p>0$, 
and let $\phi \in \Aut _kk[x_1,x_2]$ be of order $p$. 
Then, 
there exist $X_1,X_2\in k[x_1,x_2]$ and $f\in k[X_1]\sm \zs $ 
such that $k[X_1,X_2]=k[x_1,x_2]$, 
$\phi (X_1)=X_1$ and $\phi (X_2)=X_2+f$. 
\end{thm}

\subsection{Plinth ideal}\label{sect:plinth}

To begin with, 
let $B\subset B'$ be any rings, 
$\iota :B\to B'$ the inclusion map, 
and $\phi :B\to B'$ a ring homomorphism. 
We define $\delta :=\phi -\iota :
B\ni b\mapsto \phi (b)-b\in B'$. 
Then, $\delta $ is a $B^{\phi }=\ker \delta $-linear map. 
Moreover, 
the following (1) through (6) hold:

\smallskip 

\nd (1) $\delta (B)\cap B^{\phi }$ is an ideal of $B^{\phi }$. 
Indeed, 
since $\delta (B)$ and $B^{\phi }$ are $B^{\phi }$-submodules of $B'$, 
we see that $\delta (B)\cap B^{\phi }$ is a $B^{\phi }$-submodules of $B'$, 
and hence of $B^{\phi }$.

\nd (2) 
$\delta (b^l)=\phi (b)^l-b^l
=\delta (b)\sum _{i=0}^{l-1}\phi (b)^ib^{l-1-i}$ 
holds 
for each $b\in B$ and $l\ge 1$. 

\nd (3) If $\ch B=p$ is a prime number, 
then 
$\delta (b^{p^e})=(\phi (b)-b)^{p^e}=\delta (b)^{p^e}$ 
holds for each $b\in B$ and $e\ge 0$.

\nd (4) $\delta (ab)
=(\phi (a)-a)b+\phi (a)(\phi (b)-b)
=\delta (a)b+\phi (a)\delta (b)
=\delta (a)b+(\delta (a)+a)\delta (b)$ 
for each $a,b\in B$.

\nd (5) If $B=R[b_1,\ldots ,b_n]$ 
for some subring $R$ of $B^{\phi }$ and $b_1,\ldots ,b_n\in B$, 
then 
$\delta (B)\subset \sum _{i=1}^n\delta (b_i)A$ 
holds for 
$A:=B[\delta (b_1),\ldots ,\delta (b_n)]$. 
In fact, 
using (4), 
we can prove 
$\delta (b_1^{i_1}\cdots b_n^{i_n})\in \sum _{i=1}^n\delta (b_i)A$ 
for all $i_1,\ldots ,i_n\ge 0$ by induction on $i_1+\cdots +i_n$.

\nd (6) For a $\Ga $-action $\ep :B\to B[T]$ and $a\in B^{\ep }$, 
we set $\delta :=\ep -\iota :B\to B[T]$ and 
$\delta _a:=\ep _a-\id :B\to B$. 
Then, 
we have $\delta (B)\subset TB[T]$ by (A1), 
and so $\delta _a(B)\subset aB$. 

\smallskip

Next, 
we consider the case where $B=B'$. 
For $\phi \in \Aut B$, 
we define 
\begin{equation}\label{eq:def pl phi}
\pl (\phi ):=\delta (B)\cap B^{\phi },
\quad\text{where}\quad 
\delta :=\phi -\id . 
\end{equation}
By (1), 
$\pl (\phi )$ is an ideal of $B^{\phi }$, 
which we call the {\it plinth ideal} of $\phi $.

\begin{lem}\label{lem:pl principal}
In the notation above, 
the following assertions hold. 

\nd {\rm (i)} 
If $a\in \pl (\phi )$ is not a zero-divisor of $B$ 
and $\delta (B)\subset aB$, 
then $\pl (\phi )=aB^{\phi }$.

\nd {\rm (ii)} 
Assume that $B$ is a UFD, 
and $\pl (\phi )$ is a principal ideal of $B^{\phi }$. 
If $a,b\in \pl (\phi )\sm \zs $ satisfy 
$\gcd (a,b)\in B^{\phi }$, 
then $\gcd (a,b)$ belongs to $\pl (\phi )$. 
\end{lem}
\begin{proof}
(i) 
Note that $aB^{\phi }\subset \pl (\phi )
=\delta (B)\cap B^{\phi }\subset aB\cap B^{\phi }$, 
since $a\in \pl (\phi )$ and $\delta (B)\subset aB$ by assumption. 
Hence, 
it suffices to show that $aB\cap B^{\phi }\subset aB^{\phi }$, 
i.e., 
$b\in B$ and $ab\in B^{\phi }$ imply $b\in B^{\phi }$. 
Since $ab\in B^{\phi }$ and 
$a\in \pl (\phi )\subset B^{\phi }$, 
we have $ab=\phi (ab)=a\phi (b)$. 
Since $a$ is not a zero-divisor, 
it follows that $b=\phi (b)$.

(ii) Choose $c\in \pl (\phi )$ 
with $\pl (\phi )=cB^{\phi }$. 
Since $a,b\in \pl (\phi )\sm \zs $, 
we have $c\ne 0$, and 
$a=ca'$ and $b=cb'$ for some $a',b'\in B^{\phi }$. 
Then, 
we get 
$\gcd (a,b)=c\gcd (a',b')$. 
Since $\gcd (a,b)\in B^{\phi }$ by assumption, 
and $c\in \pl (\phi )$, 
this implies $\gcd (a',b')\in B^{\phi }$ 
as in the proof of (i). 
Hence, 
$\gcd (a,b)=c\gcd (a',b')$ belongs to $cB^{\phi }=\pl (\phi )$. 
\end{proof}

\begin{example}\rm
In the situation of Lemma~\ref{lem:1var}, 
we have $\pl (\phi )
=aR[x^p-a^{p-1}x]$ by Lemma~\ref{lem:pl principal}  (i), 
since 
$\delta (x)=\phi (x)-x=a\in \pl (\phi )$, 
and $\delta (R[x])\subset aR[x]$ by (5). 
\end{example}

\begin{rem}\label{rem:pl for ch order autom}\rm 
Let $B$ be a domain with $\ch B=p>0$ 
and $\phi \in \Aut B$ of order $p$.

\nd (i) It is well known that 
$B=\bigoplus _{i=0}^{p-1}B^{\phi }s^i$ 
holds for every $s\in B$ with 
$\phi (s)=s+1$ 
(cf.~e.g., \cite[Lemma 2.4]{Tani} 
for a proof using a pseudo-derivation). 
Here is another proof: 
$\phi $ extends to an automorphism of $Q(B)$ of order $p$. 
Since $[Q(B):Q(B)^{\phi }]=p$ and $s\not\in Q(B)^{\phi }$, we get 
$Q(B)=Q(B)^{\phi }(s)=\bigoplus _{i=0}^{p-1}Q(B)^{\phi }s^i$. 
Now, suppose that $B\ne \bigoplus _{i=0}^{p-1}B^{\phi }s^i$, 
and pick 
$b\in B\sm \bigoplus _{i=0}^{p-1}B^{\phi }s^i$. 
Then, 
since $b\in Q(B)$, 
we can write $b=\sum _{i=0}^{p-1}b_is^i$, 
where $b_i\in Q(B)^{\phi }$. 
Subtracting $b_is^i\in B$ from $b$ if $b_i\in B^{\phi }$, 
we may assume that 
$b=\sum _{i=0}^lb_is^i$ and $b_l\not\in B^{\phi}$ 
for some $1\le l<p$. 
Choose $b$ with least $l$. 
Then, noting $\delta (s^i)=(s+1)^i-s^i=is^{i-1}+\cdots $, 
we can write 
$\delta (b)=\sum _{i=0}^lb_i\delta (s^i)=lb_ls^{l-1}+
\sum _{i=0}^{l-2}b_i's^i$, 
where $b_i'\in Q(B)^{\phi }$. 
Since $\delta (b)\in \delta (B)\subset B$ 
and $lb_l\in Q(B)^{\phi }\sm B^{\phi }$, 
this contradicts the minimality of $l$.

\nd (ii) 
By (i), 
$B$ is a free $B^{\phi }$-module of rank $p$ 
if there exists $s\in B$ with $\phi (s)=s+1$. 
Even if such $s$ does not exist, 
the $B^{\phi }[1/u]$-module $B[1/u]$ is free of rank $p$ 
for all $u\in \pl (\phi )\sm \zs $. 
In fact, 
$\phi $ extends to an automorphism of $B[1/u]$ 
with $B[1/u]^{\phi }=B^{\phi }[1/u]$ 
and $\phi (t/u)=t/u+1$, 
where $t\in B$ is such that $u=\delta (t)$, 
i.e., $\phi (t)=t+u$. 
\end{rem}

\section{A rank three $\Ga $-action yields counterexamples to 
Question~\ref{question} }\label{sect:key}
\setcounter{equation}{0}

Assume that $\ch k=p>0$. 
Let $\ep $ be a $\Ga $-action on $\kx =k[x_1,x_2,x_3]$. 
For $h\in \kx ^{\ep }$, 
we define $\ep _h\in \Ex _3(k)$. 
Recall that 
$\pl (\ep _h)=\delta _h(\kx )\cap \kx ^{\ep _h}$, 
where $\delta _h:=\ep _h-\id $. 
The goal of this section is to 
prove the following theorem.

\begin{thm}\label{thm:main}
Assume that $\ep $ is of rank three, 
i.e., 
$\gamma (\kx ^{\ep })=0$. 
If $h\in \kx ^{\ep }$ satisfies the following condition $\clubsuit$, 
then we have $\gamma (\kx ^{\ep _h})=0$.

\nd 
$\clubsuit$ 
{\rm 
There exist 
$f_1,f_2\in \kx ^{\ep }$ such that 
$\pl (\ep _h)\subset f_1f_2\kx $ and 
$\trd_kk[f_1,f_2]=2$.} 

\end{thm}

Since $\trd _k\kx ^{\ep }=2$ 
by Remark~\ref{rem:Miyanishi} (iii), 
there always exist $f_1,f_2\in \kx ^{\ep }$ 
such that $\trd_kk[f_1,f_2]=2$. 
Then, 
$h:=f_1f_2$ satisfies $\clubsuit$, 
since $\pl (\ep _h)\subset \delta _h(\kx )\subset h\kx $ 
by \S \ref{sect:plinth} (6). 
Hence, 
the existence of a rank three $\Ga $-action on $\kx $ 
implies the existence of $\phi \in \Ex _3(k)$ 
with $\gamma (\kx ^{\phi })=0$ 
(cf.~Corollary~\ref{cor:rank3}).

\begin{lem}\label{thm:criterion}
Let $R$ be a domain with $\ch R=p>0$, 
and $\phi \in \Aut _RR[x_1,x_2]$ of order $p$. 
Then, for each $q\in R[x_1,x_2]$ with 
$\pl (\phi )\subset qR[x_1,x_2]$, 
there exists an $R$-subalgebra $B$ of $R[x_1,x_2]^{\phi }$ 
such that $q\in B$ and the following $(*)$ holds: 

\nd 
$(*)$ $\trd _RB=1$, 
and $B$ is factorially closed and 
algebraically closed in $R[x_1,x_2]$. 

\end{lem}
\begin{proof}
Put $K:=Q(R)$. 
Let $\tilde{\phi }\in \Aut _KK[x_1,x_2]$ be 
the extension of $\phi $. 
By Theorem~\ref{thm:Osaka}, 
there exist $X_1,X_2\in K[x_1,x_2]$ and $g\in K[X_1]\sm \zs $ 
such that $K[X_1,X_2]=K[x_1,x_2]$, 
$\tilde{\phi }(X_1)=X_1$ and $\tilde{\phi }(X_2)=X_2+g$. 
Multiplying by $X_2$ an element of $K^*$, 
we may assume that $X_2$ lies in $R[x_1,x_2]$. 
Now, set $B:=K[X_1]\cap R[x_1,x_2]$. 
Then, we have $\trd _RB=1$; 
$B\subset R[x_1,x_2]^{\phi }$, 
since $\tilde{\phi }(X_1)=X_1$; 
and 
$B$ is factorially closed and algebraically closed in $R[X_1,X_2]$, 
since so is $K[X_1]$  in $K[x_1,x_2]$.

From $g\in K[x_1]$, $X_2\in R[x_1,x_2]$ and 
$\phi (X_2)=X_2+g$, 
we see that $g=\phi (X_2)-X_2$ is in $B$. 
Since 
$B\subset R[x_1,x_2]^{\phi }$, 
this also shows $g\in \pl (\phi )$. 
Hence, 
we have $g=qh$ 
for some $h\in R[x_1,x_2]$, 
since $\pl (\phi )\subset qR[x_1,x_2]$ by assumption. 
Note that $qh=g$ is in $B$. 
Since $B$ is factorially closed in $R[x_1,x_2]$, 
it follows that $q$ lies in $B$. 
\end{proof}

The following lemma is true for any field $k$, 
and is readily verified.

\begin{lem}\label{lem:acfc}
Let $A$ be a $k$-domain, 
and let $B$ and $B'$ be 
$k$-subalgebras of $A$. 
Assume that $r:=\trd _kB=\trd _kB'<\infty $, 
and $B$ and $B'$ are 
factorially closed and algebraically closed in $A$. 
If there exist $a_1,\ldots ,a_r\in A$ 
such that the product $a_1\cdots a_r$ belongs to 
$B$ and $B'$, 
and $\trd _kk[a_1,\ldots ,a_r]=r$, 
then we have $B=B'$. 
\end{lem}
\begin{proof}
Both $B$ and $B'$ are equal to 
the algebraic closure of 
$k[a_1,\ldots ,a_r]$ in $A$. 
\end{proof}

\begin{proof}[Proof of Theorem~{\rm \ref{thm:main}}] 
Suppose that $\gamma (\kx ^{\ep _h})\ge 1$, 
and let $\sigma \in \Aut _k\kx $ be such that 
$R:=\sigma (k[x_1])\subset \kx ^{\ep _h}$. 
Then, 
$\ep _h$ is viewed as an element of $\Aut _RR[y_2,y_3]$, 
where $y_i:=\sigma (x_i)$. 
Since $\ep _h$ is exponential, 
$\ep _h$ is of order $p$. 
Moreover, 
we have 
$\pl (\ep _h)\subset f_1f_2\kx $ 
by the assumption $\clubsuit$. 
Hence, by Lemma~\ref{thm:criterion}, 
there exists an $R$-subalgebra 
$B$ of $\kx ^{\ep _h}$ such that $f_1f_2\in B$; 
$\trd _RB=1$, 
i.e., $\trd _kB=2$; 
and $B$ is factorially closed and 
algebraically closed in $\kx $. 
By Remark~\ref{rem:Miyanishi}, 
$\kx ^{\ep }$ is also factorially closed and algebraically closed in $\kx $, 
and $\trd _k\kx ^{\ep }=2$. 
We also have $f_1f_2\in \kx ^{\ep }$, 
and $\trd _kk[f_1,f_2]=2$ by $\clubsuit$. 
Thus, 
we get $B=\kx ^{\ep }$ by Lemma~\ref{lem:acfc}. 
Since $\sigma (k[x_1])=R\subset B$, 
this contradicts that $\gamma (\kx ^{\ep })=0$. 
\end{proof}

\section{A family of $\Ga $-actions}\label{sect:rank3 family}
\setcounter{equation}{0}

In Sections~\ref{sect:rank3 family} and 
\ref{sect:invariant ring}, 
we construct a family of $\Ga $-actions $\ep $ 
on $\kx =k[x_1,x_2,x_3]$ of rank three, 
and study $\ep _h\in \Ex _3(k)$ 
for $h\in \kx ^{\ep }$.

\subsection{Construction of the $\Ga $-actions}\label{sect:rank3 construction}

For the moment, 
let $k$ be any field with $\ch k=p\ge 0$. 
We fix 
$l,m\ge 1$ and $t\ge 2$ with $mt\ge 3$. 
We define $f:=x_1x_3-x_2^t$, 
$r:=f^lx_2+x_1^m$ and 
\begin{align}\label{eq:rank3 g}
g&:=x_1^{-1}(f^{lt+1}+r^t)
=x_1^{-1}
(f^{lt}(x_1x_3-x_2^t)+(f^lx_2+x_1^m)^t) \\ 
&=f^{lt}x_3+g^*+x_1^{mt-1},
\text{ where }
g^*:=x_1^{-1}((f^lx_2+x_1^m)^t-f^{lt}x_2^t-x_1^{mt}). 
\notag 
\end{align}
We note the following:

\nd {\bf 1}$^\circ $ 
$x_1=g^{-1}(f^{lt+1}+r^t)$, 
$x_2=f^{-l}(r-x_1^m)$ 
and $x_3=f^{-lt}(g-g^*-x_1^{mt-1})$.

\nd {\bf 2}$^\circ $ 
$g^*$ lies in $f^lx_2k[f^lx_2,x_1]$. 
If $p\mid t$, 
then $g^*$ lies in $(f^lx_2)^px_1k[(f^lx_2)^p,x_1]$. 
If $t$ is a power of $p$, 
then $g^*=0$.

Now, we set $C:=k[f^{\pm 1},g^{\pm 1}]$. 
Here, 
$h^{\pm 1}$ stands for $h,h^{-1}$ for $h\in \kx \sm \zs $. 
Then, from 1$^\circ $ and 2$^\circ $, 
we see that $\kx \subset C[r]$. 
This implies that

\nd {\bf 3}$^\circ $ 
$f$, $g$ and $r$ are algebraically independent over $k$.

\nd 
Hence, $C[r]$ is the polynomial ring in $r$ over $C$. 
Therefore, 
by Example~\ref{example:action on A[x]}, 
$$
\widetilde{\ep }:C[r]\ni u(r)\mapsto u(r+f^lgT)\in C[r][T]
$$
is a $\Ga $-action 
on $C[r]$ with $C[r]^{\widetilde{\ep }}=C$. 
Moreover, 
we have $\widetilde{\ep }(\kx )\subset \kx [T]$ 
by Proposition~\ref{prop:restricts} (i) below. 
Hence, 
$\widetilde{\ep }$ restricts to a $\Ga $-action 
on $\kx $, 
which we denote by $\ep $. 
As shown in \S \ref{sect:intersection lemma}, 
$\ep $ is of rank three 
and $\kx ^{\ep }=C\cap \kx =k[f,g]$. 
When $p=0$, $t=2$ and $m=2l+1$, 
this $\Ga $-action 
is the same as Freudenburg~\cite{Frank}.

\begin{prop}\label{prop:restricts}
Set $\delta :=\widetilde{\ep }-\iota $ {\rm (cf.~\S \ref{sect:plinth})}. 
Then, 
the following assertions hold.

\nd {\rm (i)} 
For the ideal $J:=(x_1,x_2^{lt}x_2,x_2^{lt}x_3)$ of $\kx [T]$, 
we have $\delta (\kx )\subset TJ$. 

\nd {\rm (ii)} 
If $p>0$ and $p\mid t$, 
then we have $\delta (\kx )\subset gT\kx [T]$. 
\end{prop}

Note that 
$\delta (\kx )\subset 
\sum _{i=1}^3\delta (x_i)\kx [\delta (x_1),\delta (x_2),\delta (x_3)]
\cap TC[r][T]$ 
by \S \ref{sect:plinth} (5) and (6). 
Hence, 
Proposition~\ref{prop:restricts} follows 
from (i) and (ii) of the following lemma.

\begin{lem}\label{lem:rank3 J}
{\rm (i)} 
$\delta (x_i)\in J$ holds for $i=1,2,3$.

\nd {\rm (ii)} 
If $p>0$ and $p\mid t$, 
then $\delta (x_i)\in g\kx [T]$ holds for $i=1,2,3$.

\nd {\rm (iii)} 
If $p>0$, 
and $m$ and $t$ are powers of $p$, 
then 
we have $\delta (x_1)=g^{-1}(f^lgT)^t$, 
$\delta (x_2)=gT-f^{-l}\delta (x_1)^m$, 
and $\delta (x_3)=-f^{-lt}\delta (x_1^{mt-1})$.

\end{lem}

The rest of \S \ref{sect:rank3 construction} 
is devoted to the proof of this lemma. 
First, note the following: 

\nd {\bf 4}$^\circ $ 
Since $g\in x_1^{mt-1}+f\kx $ by (\ref{eq:rank3 g}) and 2$^\circ $, 
we have $\gcd (f,g)=\gcd (f,x_1^{mt-1})=1$.

\nd {\bf 5}$^\circ $ 
Since $f\in (x_1,x_2^t)$, 
we have $f^lx_i\in (x_1,x_2^{lt}x_i)\subset J$ 
for $i=1,2,3$. 
Hence, $f^l\fn \subset J$ 
holds for $\fn :=(x_1,x_2,x_3)$.

\nd {\bf 6}$^\circ $ 
$g^*$ is in $(f^lx_2{\cdot }x_1)$ if $m\ge 2$, 
and in $((f^lx_2)^2,f^lx_2{\cdot }x_1)$ if $m=1$ and $t\ge 3$. 
Hence, 
$g^*$ lies in $J^2$ by 5$^\circ $. 
Thus, 
we have $g\in (f^lf^{l}x_3,g^*,x_1^2)\subset f^lJ+J^2$.

\nd {\bf 7}$^\circ $ 
$r\in (f^lx_2,x_1)\subset J$ by 5$^\circ $, 
and so 
$\widetilde{\ep }(r)=r+f^lgT\in (r,g)\subset J$ 
by 6$^\circ $. 
Hence, we know 
by \S \ref{sect:plinth} (2) that 
$$
\delta (r^u)
=\delta (r)\sum _{i=0}^{u-1}\widetilde{\ep }(r)^ir^{u-1-i}
=f^lgT\sum _{i=0}^{u-1}\widetilde{\ep }(r)^ir^{u-1-i}\in f^lgJ^{u-1}
\text{ for all }u\ge 1. 
$$

\begin{proof}[Proof of Lemma~$\ref{lem:rank3 J}$]
(i) 
$\delta $ is a 
linear map over $\ker \delta =C[r]^{\widetilde{\ep }}=C$ 
(cf.~\S \ref{sect:plinth}). 
Hence,

\nd {\bf 8}$^\circ $ 
we have 
$\delta (x_1)=\delta (g^{-1}(f^{lt+1}+r^t))=
g^{-1}\delta (r^t)\in f^lJ^{t-1}$ 
by 1$^\circ $ and 7$^{\circ }$.

\nd 
Since $t\ge 2$, 
this proves $\delta (x_1)\in J$. 
Hence, 
$\widetilde{\ep }(x_1)=x_1+\delta (x_1)$ is in $J$. 
Moreover,

\nd {\bf 9}$^\circ $ 
$\delta (x_1)\fn \subset f^l\fn \cdot J^{t-1}\subset J^2$ 
holds by 5$^\circ $ and 8$^\circ $.

Similarly, 
by 1$^\circ $ and \S \ref{sect:plinth} (2), 
we have 
\begin{equation}\label{eq:delta(x_2)}
\delta (x_2)
=f^{-l}(\delta (r)-\delta (x_1^m))
=gT-f^{-l}\delta (x_1)\sum _{i=0}^{m-1}\widetilde{\ep }(x_1)^ix_1^{m-1-i},
\end{equation}
in which $g\in f^lJ+J^2$ by 6$^\circ $, 
$f^{-l}\delta (x_1)\in J^{t-1}$ by 8$^\circ $, 
and $\widetilde{\ep }(x_1),x_1\in J$. 
Hence, 
$\delta (x_2)$ is in $f^lJ+J^2+J^{t+m-2}$. 
Since $mt\ge 3$, 
we have $t+m\ge 4$. 
Thus, 
we get $\delta (x_2)\in f^lJ+J^2\subset J$. 
This implies 
$\widetilde{\ep }(x_2)=x_2+\delta (x_2)\in \fn $, 
and so

\nd {\bf 10}$^\circ $ 
$\delta (x_2^t)=\delta (x_2)\sum _{i=0}^{t-1}\widetilde{\ep }(x_2)^ix_2^{t-1-i}
\in (f^lJ+J^2)\mathfrak{n}\subset J^2$ 
by 5$^\circ $, 
since $t\ge 2$.

For $\delta (x_3)$, 
first note that 
$\delta (k[x_1,x_2,f])\subset 
\sum _{i=1}^2\delta (x_i)k[x_1,x_2,f,\delta (x_1),\delta (x_2)]$ 
by \S \ref{sect:plinth} (5), 
since $\delta (f)=0$. 
By 2$^\circ $, 
$g^*+x_1^{mt-1}$ is in $k[x_1,x_2,f]$. 
Hence, 
we get

\nd {\bf 11}$^\circ $ 
$\delta (x_3)=-f^{-lt}\delta (g^*+x_1^{mt-1})
\in f^{-lt}\sum _{i=1}^2\delta (x_i)k[x_1,x_2,f,\delta (x_1),\delta (x_2)]$ 
by 1$^\circ $. 

\nd 
We have already proved that 
$\delta (x_1),\delta (x_2)\in J\subset \kx [T]$. 
Thus, 
$\delta (x_3)$ belongs to $f^{-lt}\kx [T]$ by 11$^\circ $. 
Since 
$0=\delta (f)
=\delta (x_1)x_3+\widetilde{\ep }(x_1)\delta (x_3)-\delta (x_2^t)$ 
by \S \ref{sect:plinth} (4), 
we have 
$\widetilde{\ep }(x_1)\delta (x_3)
=\delta (x_2^t)-\delta (x_1)x_3\in \kx [T]$. 
This implies $\delta (x_3)\in \kx [T]$, 
since $\widetilde{\ep }(x_1)=x_1+\delta (x_1)\in x_1+fJ^{t-1}$ by 8$^\circ $, 
and hence $\gcd (f,\widetilde{\ep }(x_1))=\gcd (f,x_1)=1$.

Suppose that $\delta (x_3)\not\in J$. 
Then, 
there appears in $\delta (x_3)$ a monomial $h=x_2^{i_2}x_3^{i_3}T^j$ 
not in $J$. 
For such $h$, 
we have $x_1h\not\in J^2$. 
Since $J^2$ is a monomial ideal, 
this means that the monomial 
$x_1h$ does not appear in any polynomial belonging to $J^2$. 
We choose $h$ so that $i_2+i_3$ is minimal. 
Now, observe that 
\begin{equation}\label{eq:lem J delta(x_3)}
\delta (x_2^t)-\delta (x_1)x_3
=\widetilde{\ep }(x_1)\delta (x_3)=x_1\delta (x_3)+\delta (x_1)\delta (x_3). 
\end{equation}
By 9$^\circ $ and 10$^\circ $, 
$\delta (x_2^t)-\delta (x_1)x_3$ lies in $J^2$. 
Hence, 
$x_1h$ does not appear in (\ref{eq:lem J delta(x_3)}). 
Clearly, 
$x_1h$ appears in $x_1\delta (x_3)$. 
Thus, 
$x_1h$ must appear in $\delta (x_1)\delta (x_3)$. 
This implies $i_2+i_3\ge 1$, 
since $\delta (x_1)\delta (x_3)\in (f)\subset \fn ^2$ by 8$^\circ $. 
Moreover, $\delta (x_3)$ is not in $\fn $, 
for otherwise $\delta (x_1)\delta (x_3)\in J^2$ by 9$^\circ $. 
Hence, 
the monomial $T^{j'}$ 
appears in $\delta (x_3)$ for some $j'\ge 0$. 
This contradicts the minimality of $i_2+i_3$.

(ii) 
We have $\delta (x_1)=g^{-1}\delta (r^t)=g^{-1}\delta (r^{t/p})^p
\in g^{-1}(f^lg)^p\kx [T]\subset f^lg\kx [T]$ 
by 8$^\circ $, \S \ref{sect:plinth} (3) and 7$^\circ $. 
This implies $\delta (x_1)\in g\kx [T]$, 
and also $\delta (x_2)\in g\kx [T]$ by (\ref{eq:delta(x_2)}). 
Then, 11$^\circ $ yields that 
$\delta (x_3)\in f^{-lt}g\kx [T]$. 
Since $\delta (x_3)\in \kx [T]$ by (i), 
and $\gcd (f,g)=1$  by 4$^\circ $, 
it follows that $\delta (x_3)\in g\kx [T]$.

(iii) 
For $\delta (x_1)$ and $\delta (x_2)$, 
apply \S \ref{sect:plinth} (3) to 
8$^\circ $ and (\ref{eq:delta(x_2)}). 
Since $g^*=0$ by 2$^\circ $, 
we have 
$\delta (x_3)=-f^{-lt}\delta (g^*+x_1^{mt-1})
=-f^{-lt}\delta (x_1^{mt-1})$ by 1$^\circ $. 
\end{proof}

\subsection{Plinth ideals}\label{sect:rank3 exp}

Assume that $p>0$. 
Let $\ep $ be the $\Ga $-action on $\kx $ 
as in \S \ref{sect:rank3 construction}. 
For $h\in \kx ^{\ep }\sm \zs $, 
we define $\ep _h\in \Ex _3(k)$, 
and a $\kx ^{\ep _h}$-linear map 
$\delta _h:=\ep _h-\id $. 
In \S \ref{sect:rank3 exp}, 
we study the plinth ideal 
$\pl (\ep _h)=\delta _h(\kx )\cap \kx ^{\ep _h}$. 
Our goal is to prove the following theorem.

\begin{thm}\label{thm:plinth rank 3}
\nd{\rm (i)} 
If $p\nmid t$, 
then $\pl (\ep _h)$ is not a principal ideal of $\kx ^{\ep _h}$.

\nd{\rm (ii)} 
If $p\mid t$, 
then we have $\pl (\ep _h)=gh\kx ^{\ep _h}$. 
\end{thm}

By definition, 
we have $\ep _h(r)=\ep (r)|_{T=h}=r+f^lgh$, 
and $\delta _h(r)=\ep _h(r)-r=f^lgh$. 
Since $f,g,h\in \kx ^{\ep }\subset \kx ^{\ep _h}$, 
it follows that 

\nd {\bf 12}$^\circ $  
$f^lgh$ belongs to $\pl (\ep _h)$. 
Consequently, 
we have $f^lgh\kx ^{\ep _h}\subset \pl (\ep _h)$. 

\begin{lem}\label{lem:rank3 pl}
If $p\nmid t$, 
then 
$f^{l'}h$ belongs to $\pl (\ep _h)$ for some $l'\ge l$. 
\end{lem}

\begin{proof}
By Euler's theorem, 
$p^v\equiv 1\pmod{t}$ holds for $v:=|(\Z /t\Z )^*|$, 
since $p\nmid t$ by assumption. 
Set $u:=(p^v-1)/t$. 
Then, 
since $r^t+f^{lt+1}=x_1g$ by (\ref{eq:rank3 g}), 
we have 
$$
s:=r^{p^v}-(-f^{lt+1})^ur=(r^{tu}-(-f^{lt+1})^u)r
\in (r^t-(-f^{lt+1}))k[r,f]\subset g\kx . 
$$
Hence, 
$\delta _h(g^{-1}s)$ is in $\delta _h(\kx )$. 
On the other hand, 
by \S \ref{sect:plinth} (3), 
we have 
$g\delta _h(g^{-1}s)=\delta _h(s)
=\delta _h(r)^{p^v}-(-f^{lt+1})^u\delta _h(r)
=(f^lgh)^{p^v}-(-f^{lt+1})^uf^lgh$. 
This gives that 
\begin{align}\label{eq:(Z/tZ)^*}
\delta _h(g^{-1}s)=g^{p^v-1}(f^lh)^{p^v}-(-f^{lt+1})^uf^lh 
\in k[f,g,h]\subset \kx ^{\ep _h}. 
\end{align}
Thus, 
(\ref{eq:(Z/tZ)^*}) belongs to $\pl (\ep _h)$. 
In the right-hand side of (\ref{eq:(Z/tZ)^*}), 
$g^{p^v-1}(f^lh)^{p^v}$ belongs to $\pl (\ep _h)$ 
by 12$^\circ $. 
Therefore, 
$(f^{lt+1})^uf^lh$ belongs to $\pl (\ep _h)$. 
\end{proof}

Now, we set 
$I:=x_1\kx +x_2^{lt}x_2\kx +x_2^{lt}x_3\kx $. 
Then, by Proposition~\ref{prop:restricts},

\nd {\bf 13}$^\circ $ 
we have $\delta _h(\kx )\subset hI$. 
Moreover, if $p\mid t$, 
then $\delta _h(\kx )\subset gh\kx $.

\begin{proof}[Proof of Theorem {\rm \ref{thm:plinth rank 3} (i)}]
By Lemma~\ref{lem:rank3 pl} and 12$^\circ $, 
$f^{l'}h$ and $f^lgh$ belong to $\pl (\ep _h)$, 
where $l'\ge l$. 
Moreover, 
we have $\gcd (f^{l'}h,f^lgh)=f^lh$ by 4$^\circ $, 
and $f^lh$ is in $\kx ^{\ep _h}$. 
Hence, 
if $\pl (\ep _h)$ is principal, 
then $f^lh$ must lie in $\pl (\ep _h)$ 
by Lemma~\ref{lem:pl principal} (ii). 
Since $\pl (\ep _h)\subset \delta _h(\kx )\subset hI$ by 13$^\circ $, 
it follows that $f^lh\in hI$, 
and so $f^l\in I$. 
Since $x_1\in I$, 
this implies that $x_2^{lt}\in I$, 
a contradiction. 
\end{proof}

To prove (ii), 
we need to construct some elements of $\kx ^{\ep _h}$. 
The following remark is also used in 
Section~\ref{sect:Nagata}.

\begin{notation}\label{notation:taylor}\rm 
Let $R$ be a ring, and $P(T)\in R[T]$. 
For each $i\ge 0$, 
we define $P_i(T)\in R[T]$ by $P(T+U)=\sum _{i\ge 0}P_i(T)U^i$. 
We note that 
$P_0(T)=P(T)$ and $P_1(T)=P'(T):=dP(T)/dT$ 
regardless of the characteristic of $R$. 
\end{notation}

\begin{rem}\label{rem:q_1}\rm
Let $B$ be a domain with $\ch B=p>0$. 
Assume that 
$\phi \in \Aut B$, 
$a,b\in B^{\phi }\sm \zs $ and $c\in B$ 
satisfy $\phi (c)=c+ab$. 
Then, 
$B^{\phi }$ contains $q:=c^p-(ab)^{p-1}c$ (cf.~(\ref{eq:x^p-ax})). 
Now, let 
$S$ be a subring of $B^{\phi }$ 
and let $\xi (T)=\sum _{i\ge 0}r_iT^i\in S[T]$, 
where $r_i\in S$. 
If $\xi (c)=aw$ for some $w\in B$, 
then the following (i) and (ii) hold.

\nd (i) 
Set $\xi ^p(T):=\sum _{i\ge 0}r_i^pT^i$. 
Then, we have $\xi ^p(c^p)=\xi (c)^p=(aw)^p$ 
and $(\xi ^p)_1(c^p)
=\sum _{i\ge 0}ir_i^pc^{(i-1)p}
=\xi '(c)^p$. 
Hence, we get 
\begin{align}\label{eq:q_1}
\widetilde{q}_1:=a^{1-p}\xi ^p(q)
&=a^{1-p}\Bigl(
\xi ^p(c^p)
-(\xi ^p)_1(c^p){\cdot }(ab)^{p-1}c
+\sum _{i\ge 2}(\xi ^p)_i(c^p){\cdot }
(-(ab)^{p-1}c)^i
\Bigr)
\notag \\
&\quad 
\in aw^p-\xi '(c)^p{\cdot } 
b^{p-1}c+a^{p-1}b^{2(p-1)}c^2S[ab,c]
\subset B. 
\end{align}
Since $a,\xi ^p(q)\in B^{\phi }$, 
it follows that $\widetilde{q}_1\in B^{\phi }$.

\nd (ii) 
Assume that $\xi (T)=\xi ^*(T^p)-a\widehat{\xi }(T)$ 
for some $\xi ^*(T),\widehat{\xi }(T)\in S[T]$. 
Then, 
we have $\xi ^*(c^p)=\xi (c)+a\widehat{\xi }(c)=a(w+\widehat{\xi }(c))$. 
Hence, we get 
\begin{equation}\label{eq:q_1^*}
\begin{aligned}
q_1&:=a^{-1}\xi ^*(q)
=a^{-1}\Bigl(
\xi ^*(c^p)+\sum _{i\ge 1}(\xi ^*)_i(c^p){\cdot }(-(ab)^{p-1}c)^i
\Bigr) \\
&\quad \in w+\widehat{\xi }(c)+a^{p-2}b^{p-1}cS[ab,c]\subset B. 
\end{aligned}
\end{equation}
Thus, $q_1$ belongs to $B^{\phi }$ as in (i). 
\end{rem}

Now, 
observe that 
$\ep _h(r)=r+g{\cdot }f^lh$ and $f^{lt+1}+r^t=g{\cdot }x_1$. 
Hence, 
we can use Remark~\ref{rem:q_1} for 
$(a,b,c,S,\xi (T),w) =(g,f^lh,r,k[f],f^{lt+1}+T^t,x_1)$. 
Thus, 
\begin{equation}\label{eq:rank3 q}
q:=r^p-(f^lgh)^{p-1}r
=(f^{lp}x_2^p+x_1^{mp})-(f^lgh)^{p-1}(f^lx_2+x_1^m) 
\end{equation}
belongs to $\kx ^{\ep _h}$. 
Moreover, 
if $p\mid t$, 
then 
we can write $f^{lt+1}+T^t=\xi ^*(T^p)-g\widehat{\xi }(T)$, 
where $\xi ^*(T):=f^{lt+1}+T^{t/p}$ 
and $\widehat{\xi }(T):=0$. 
Hence, 
by (\ref{eq:q_1^*}), 
$\kx ^{\ep _h}$ contains 
\begin{equation}\label{eq:rank3 q_1}
q_1:=g^{-1}(f^{lt+1}+q^{t/p})
\in x_1+g^{p-2}(f^lh)^{p-1}rk[f,f^lgh,r].
\end{equation}

\begin{proof}[Proof of Theorem {\rm \ref{thm:plinth rank 3} (ii)}]
Since $p\mid t$ by assumption, 
we have $\delta _h(\kx )\subset gh\kx $ by 13$^\circ $. 
Hence, by Lemma~\ref{lem:pl principal} (i), 
it suffices to show that $gh$ belongs to $\pl (\ep _h)$. 
Since $r\equiv x_1^m$, $q_1\equiv x_1\pmod{f^l\kx }$, 
we have $s:=r-q_1^m\in f^l\kx $. 
Hence, 
$\delta _h(f^{-l}s)$ is in $\delta _h(\kx )$. 
On the other hand, 
we have $f^l\delta _h(f^{-l}s)=\delta _h(s)=\delta _h(r)=f^lgh$, 
since $q_1\in \kx ^{\ep _h}=\ker \delta _h$. 
Thus, 
$\delta _h(f^{-l}s)=gh$ is in $\kx ^{\ep _h}$. 
This proves 
$gh\in \pl (\ep _h)$.
\end{proof}

\section{Invariant ring}\label{sect:invariant ring}
\setcounter{equation}{0}

Let $\ep $ be the $\Ga $-action on $\kx =k[x_1,x_2,x_3]$ 
defined in \S \ref{sect:rank3 construction}. 
In this section, 
we prove the following three theorems. 
Theorem~\ref{thm:k[f,g]} holds for any $k$.

\begin{thm}\label{thm:k[f,g]}
We have $\kx ^{\ep }=k[f,g]$, 
and $\ep $ is of rank three, i.e., 
$\gamma (\kx ^{\ep })=0$. 
\end{thm}

Now, assume that $p>0$. 
For $0\ne h\in \kx ^{\ep }=k[f,g]$, 
we define $\ep _h\in \Ex _3(k)$. 
The following corollary is a consequence of Theorems~\ref{thm:main} 
and \ref{thm:k[f,g]}. 
The case (1) is the same as the remark after Theorem~\ref{thm:main}, 
but we use 13$^\circ $ for the case (2). 

\begin{cor}\label{cor:rank3}
We have $\gamma (\kx ^{\ep _h})=0$ if $p>0$ and one of the following holds. 

\noindent{\rm (1)} 
$h=f_1f_2$ for some $f_1,f_2\in k[f,g]$ 
with $\trd _kk[f_1,f_2]=2$.  

\noindent{\rm (2)} 
$p\mid t$ and $h\in k[f,g]\sm k[g]$. 
\end{cor}

Note that $\ep _h$ is the restriction of 
$\widetilde{\ep }_h:C[r]\ni u(r)\mapsto u(r+f^lgh)\in C[r]$ to $\kx $. 
By Lemma~\ref{lem:1var}, 
we have $C[r]^{\widetilde{\ep }_h}=C[q]$, 
where $q$ is as in (\ref{eq:rank3 q}). 
Hence, 
we get 
\begin{equation}\label{eq:C[r]^{ep_h}}
\begin{aligned}
&\kx ^{\ep _h}=C[r]^{\widetilde{\ep }_h}\cap \kx =C[q]\cap \kx 
=k[f^{\pm 1},g^{\pm 1},q]\cap \kx .
\end{aligned}
\end{equation}
To describe $\kx ^{\ep _h}$, 
we use the isomorphism 
\begin{equation}\label{eq:psi}
\psi :C[r^p]\ni u(r^p)\mapsto 
u(r^p-(f^lgh)^{p-1}r)=u(q)\in C[q]=C[r]^{\widetilde{\ep }_h}.
\end{equation}
By 1$^\circ $ and 2$^\circ $, 
it is easy to check that 
$x_i^p\in C[r^p]$ for $i=1,2,3$. 
If $p\mid t$, 
then we also have 
$x_1,x_3\in C[r^p]$. 
Moreover, 
$\psi (x_1)=\psi (g^{-1}(f^{lt+1}+r^t))=g^{-1}(f^{lt+1}+q^{t/p})$ 
is equal to $q_1$ in (\ref{eq:rank3 q_1}). 
With this notation, 
the following theorems hold.

\begin{thm}\label{thm:rank3 invariant ring}
Assume that $p\mid t$.

\nd 
{\rm (i)} 
There exist $q_2,q_3\in \kx ^{\ep _h}$ 
such that $q_1q_3-q_2^{t/p}=f$ 
and $\kx ^{\ep _h}=k[q_1,q_2,q_3]$.

\nd {\rm (ii)} 
We have 
$\psi (k[x_1,x_2^p,x_3])=\kx ^{\ep _h}$ 
if and only if $h^{p-1}\in f^lk[f,g]$. 

\end{thm}

Let $k[\x,y,z]=\kx [y,z]$ 
be the polynomial ring in five variables over $k$. 

\begin{thm}\label{thm:rank3 invariant ring2}
Assume that $p\nmid t$, $p\nmid mt-1$ 
and $h^{p-1}\in f^{l+1}g^2k[f,g]\sm \zs $.

\nd 
{\rm (i)} 
We have $\kx ^{\ep _h}=\psi (k[x_1^p,x_2^p,x_2^p,f,g])
=k[\psi (x_1^p),\psi (x_2^p),\psi (x_3^p),f,g]$.

\nd 
{\rm (ii)} 
The $k$-algebra $\kx ^{\ep _h}$ 
is isomorphic to $k[\x,y,z]/(y^p-f,z^p-g)$, 
and is not isomorphic to $\kx $. 
\end{thm}

\subsection{Intersection lemmas}\label{sect:intersection lemma}

First, we prove some lemmas. 
If $I$ is an ideal of a ring $B$, 
then $\pi :B\to B/I$ denotes the natural surjection. 
Now, let $R\subset B$ be domains, 
$a_1,\ldots ,a_s,b\in B\sm \zs $ with $b\not\in B^*$, 
and $\Ry =R[y_1,\ldots ,y_s]$ 
the polynomial ring in $s$ variables over $R$. 
For $\sigma :\Ry \ni u(\y )\mapsto u(a_1,\ldots ,a_s)\in B$ 
and $\pi :B\to B/bB$, 
we set $\ol{\sigma }:=\pi \circ \sigma $ 
and $A:=\sigma (\Ry )=R[a_1,\ldots ,a_s]$. 
Then, the following lemma holds. 

\begin{lem}\label{lem:intersection}
If there exists $\mathcal{S}\subset \Ry $ 
such that $\ker \ol{\sigma }=(\mathcal{S})$ and 
$\sigma (\mathcal{S})\subset bA[b]$, 
then we have $A[
b^{\pm 1}]\cap B=A[b]$. 
If moreover $b$ is in $A$, 
then $A[b^{-1}]\cap B=A$. 
\end{lem}
\begin{proof}
First, 
we show that $A\cap bB\subset bA[b]$. 
Pick any $a\in A\cap bB$. 
Since $a\in A$, 
we can write $a=\sigma (p)$, 
where $p\in \Ry $. 
Then, 
we have 
$p\in \ker \ol{\sigma }=(\mathcal{S})$, 
since $\sigma (p)=a\in bB$. 
Write $p=\sum _ip_iq_i$, 
where $p_i\in \mathcal{S}$ and $q_i\in \Ry $. 
Since $\sigma (p_i)\in bA[b]$ by assumption, 
and $\sigma (q_i)\in A$, 
it follows that 
$a=\sigma (p)=\sum _i\sigma (p_i)\sigma (q_i)\in bA[b]$.

Now, we prove that $A[b^{\pm 1}]\cap B\subset A[b]$ by contradiction. 
Suppose that there exists $c\in A[b^{\pm 1}]\cap B\sm A[b]$. 
Choose the least $u\ge 1$ with $cb^u\in A[b]$, 
and write $cb^u=\sum _{i\ge 0}c_ib^i$, 
where $c_i\in A$. 
Then, 
$c_0$ is not in $bA[b]$ by the minimality of $u$. 
On the other hand, 
$c_0=cb^u-\sum _{i\ge 1}c_ib^i$ belongs to $A\cap bB\subset bA[b]$ 
by the discussion above. 
This is a contradiction. 
The inclusion $A[b^{\pm 1}]\cap B\supset A[b]$ is clear. 
\end{proof}

\begin{rem}\label{rem:indep}\rm 
Regard $B/bB$ as an $R$-algebra. 
If $\ol{a}_1,\ldots ,\ol{a}_l\in B/bB$ are algebraically 
independent over $R$, 
then the assumption of Lemma~\ref{lem:intersection} 
holds with $\mathcal{S}=\zs $. 
\end{rem}

Next, 
we consider the case where $b$ is in $R$. 
Since $bR\subset R\cap bB$, 
there exists a natural homomorphism 
$\ol{R}:=R/bR\to R/(R\cap bB)\to B/bB$. 
Hence, 
we can regard $B/bB$ as an $\ol{R}$-algebra. 
We define a substitution map $\widehat{\sigma }:\ol{R}[\y ]\to B/bB$ 
by $\widehat{\sigma }(y_i)=\ol{a}_i$ for $i=1,\ldots ,s$. 
Then, the following lemma holds.

\begin{lem}\label{lem:intersection 2}
Assume that $b$ is in $R$. 
If there exists $\mathcal{S}\subset \Ry $ such that 
the image of $\mathcal{S}$ in $\ol{R}[\y ]$ generates 
$\ker \widehat{\sigma }$, 
and $\sigma (\mathcal{S})\subset bA[b]$, 
then we have $A[b^{-1}]\cap B=A$. 
\end{lem}
\begin{proof}
Note that $\ol{\sigma }$ equals $\Ry \twoheadrightarrow \ol{R}[\y ]
\stackrel{\widehat{\sigma }}{\to }B/bB$. 
Since $\ker \widehat{\sigma }$ 
is generated by the image of $\mathcal{S}$ in $\ol{R}[\y ]$, 
we see that $\ker \ol{\sigma }=(\{ b\} \cup \mathcal{S})$. 
Since $\sigma (b)=b\in bA[b]$, 
and $\sigma (\mathcal{S})\subset bA[b]$ by assumption, 
we get $A[b^{-1}]\cap B=A$ 
by Lemma~\ref{lem:intersection}. 
\end{proof}

\begin{rem}\label{rem:indep2}\rm 
Assume that $b$ is in $R$. 
If $\ol{a}_1,\ldots ,\ol{a}_l\in B/bB$ are algebraically 
independent over $R/bR$, 
then the assumption of Lemma~\ref{lem:intersection 2} 
holds with $\mathcal{S}=\zs $. 
\end{rem}

\begin{rem}\label{rem:intersection}\rm
Let $S$ be a subring of $B$, 
and $a,b\in B\sm \zs $. 
If $S[a,b^{\pm 1}]\cap B=S[a,b]$, 
then 
we have $S[a^{\pm 1},b^{\pm 1}]\cap B=S[a^{\pm 1},b]\cap B$. 
Actually, 
if $c\in S[a^{\pm 1},b^{\pm 1}]\cap B$, 
there exists $u\ge 0$ such that $a^uc\in S[a,b^{\pm 1}]\cap B=S[a,b]$. 
This implies $c\in S[a^{\pm 1},b]$. 
\end{rem}

\begin{proof}[Proof of Theorem~{\rm \ref{thm:k[f,g]}}]
Since 
$\kx ^{\ep }=
C\cap \kx =k[f^{\pm 1},g^{\pm 1}]\cap \kx $, 
we show that 
(i) $k[f^{\pm 1},g^{\pm 1}]\cap \kx =k[f^{\pm 1},g]\cap \kx $ and 
(ii) $k[
f^{\pm 1},g]\cap \kx =k[f,g]$. 
For (i), 
it suffices to check that 
(i$'$) $k[f,g^{\pm 1}]
\cap \kx =k[f,g]$ 
by Remark~\ref{rem:intersection}.

(i$'$) 
By Remark~\ref{rem:indep} 
with $(R,a_1,b)=(k,f,g)$, 
it suffices to show that 
$\ol{f}\in \kx /g\kx $ is transcendental over $k$. 
Supposing the contrary, 
there exist $\nu (T)\in k[T]\sm \zs $ 
and $H\in \kx $ such that $\nu (f)=gH$. 
Then, 
we get $\nu (-x_2^t)=g|_{x_3=0}{\cdot }H|_{x_3=0}$ 
by the substitution $x_3\mapsto 0$. 
This implies that $g|_{x_3=0}\in k[x_2]$, 
which is absurd (cf.~(\ref{eq:rank3 g})).

(ii) 
We repeat the same argument with $f$ and $g$ interchanged. 
If $\nu (g)=fH$ for some $\nu (T)\in k[T]\sm \zs $ 
and $H\in \kx $, 
then we get $\nu (x_1^{mt-1})=0$ 
by the substitution $x_2,x_3\mapsto 0$, 
since $f\mapsto 0$ and $g\mapsto x_1^{mt-1}$ by 2$^\circ $. 
This is a contradiction.

Observe that 
$f$ and $g$ have no linear part, 
since $t,mt-1\ge 2$. 
Hence, 
no element of $k[f,g]$ has a linear part. 
This implies that 
$\gamma (k[f,g])=0$ (cf.~\cite{Frank}). 
In fact, 
$\sigma (x_i)$ has a linear part for all $\sigma \in \Aut _k\kx $ and $i$, 
since the Jacobian of $\sigma $ lies in $k^*$. 
\end{proof}

\subsection{Proof of Theorem~\ref{thm:rank3 invariant ring}}
\label{sect:rank3 invariant proof}

The goal of \S \ref{sect:rank3 invariant proof} 
is to prove Theorem~\ref{thm:rank3 invariant ring}. 
For the moment, 
let $\ol{u}$ denote the image of $u\in \kx $ in $\kx /g\kx $. 
Then, 
from (\ref{eq:rank3 q}) and (\ref{eq:rank3 q_1}), 
we see that

\nd {\bf 14}$^\circ $ 
$\ol{q}=\ol{r^p-(f^lgh)^{p-1}r}=\ol{r^p}$ 
and $\ol{q}_1\in \ol{x}_1
+\ol{g^{p-2}(f^lh)^{p-1}r}k[\ol{f},\ol{r}]
\subset \ol{x}_1+k[\ol{f},\ol{r}]$.

\begin{lem}\label{lem:K_g}
$g$ is irreducible in $\kx $. 
If $p\mid t$, then 
we have $\trd _kk[\ol{f},\ol{q},\ol{q}_1]=2$. 
\end{lem}
\begin{proof}
For the first part, 
suppose that $g=p_1p_2$ for some $p_1,p_2\in \kx \sm k$. 
Then, $p_1$ and $p_2$ are in $k[f,g]$, 
since $k[f,g]=\kx ^{\ep }$ is factorially closed in $\kx $ 
by Remark~\ref{rem:Miyanishi} (i). 
Hence, we can write $g=\gamma _1(f,g)\gamma _2(f,g)$, 
where $\gamma _i(x,y)\in k[x,y]\sm k$. 
This contradicts 3$^\circ $. 
For the last part, 
note that $\trd _k\kx /g\kx =2$, 
and $k[\ol{f},\ol{q},\ol{q}_1]=k[\ol{f},\ol{r}^p,\ol{q}_1]$ 
by 14$^\circ $. 
Hence, it suffices to show that 
$Q(\kx /g\kx )$ is algebraic over 
$k(\ol{f},\ol{r}^p,\ol{q}_1)$. 
We have 
$x_2=f^{-l}(r-x_1^m)$ by 1$^\circ $, 
and $x_3=x_1^{-1}(f-x_2^t)$ since $f=x_1x_3-x_2^t$. 
Hence, $x_2$ and $x_3$ are in $k[f^{\pm 1},r,x_1^{\pm 1}]$. 
Clearly, 
$\ol{f}$ and $\ol{x}_1$ are nonzero. 
Thus, 
$Q(\kx /g\kx )=k(\ol{x}_1,\ol{x}_2,\ol{x}_3)$ 
is equal to $k(\overline{f},\overline{r},\overline{x}_1)$. 
Since $\ol{q}_1\in \ol{x}_1+k[\ol{f},\ol{r}]$ by 14$^\circ $, 
$k(\overline{f},\overline{r},\overline{x}_1)$ 
is equal to 
$k(\overline{f},\overline{r},\overline{q}_1)$, 
which is algebraic over 
$k(\overline{f},\overline{r}^p,\overline{q}_1)$. 
\end{proof}

Let $\ky =k[y_1,y_2,y_3]$ be the polynomial ring in three variables over $k$. 

\begin{prop}\label{prop:invariant ring first step}
If $p\mid t$, 
then we have $\kx ^{\ep _h}=k[f^{\pm 1},g,q,q_1]\cap \kx $. 
\end{prop}

\begin{proof}
Since $q_1$ is in $\kx ^{\ep _h}$, 
(\ref{eq:C[r]^{ep_h}}) implies that 
$\kx ^{\ep _h}=k[f^{\pm 1},g^{\pm 1},q,q_1]\cap \kx $. 
We show that this 
is equal to  
$k[f^{\pm 1},g,q,q_1]\cap \kx $. 
By Remark~\ref{rem:intersection}, 
it suffices to verify that 
$k[f,g^{\pm 1},q,q_1]\cap \kx =k[f,g,q,q_1]$, 
i.e., 
$A[g^{\pm 1}]\cap \kx =A[g]$, 
where $A:=k[f,q,q_1]$. 
For $\sigma :\ky \ni \nu (\y )\mapsto \nu (f,q,q_1)\in \kx $ 
and $\pi :\kx \to \kx /g\kx $, 
we set $\ol{\sigma }=\pi \circ \sigma $. 
Then, 
we have $\trd _k\ol{\sigma }(\ky )=
\trd _kk[\ol{f},\ol{q},\ol{q}_1]=2$ by Lemma~\ref{lem:K_g}. 
Hence, 
$\ker \ol{\sigma }$ is a prime ideal of $\ky $ of height one, 
and thus principal. 
Therefore, 
$\eta \in \ker \ol{\sigma }$ satisfies 
$\ker \ol{\sigma }=(\eta )$ 
whenever $\eta $ is irreducible in $\ky $. 
Now, 
observe that 
(1) $\sigma (y_1^{lt+1}+y_2^{t/p})=f^{lt+1}+q^{t/p}=gq_1\in gA[g]$ 
by (\ref{eq:rank3 q_1}), 
and (2) $y_1^{lt+1}+y_2^{t/p}$ is irreducible in $\ky $, 
since $\gcd (lt+1,t/p)=1$. 
(1) implies that 
$\ol{\sigma }(y_1^{lt+1}+y_2^{t/p})=0$. 
Hence, 
$\ker \ol{\sigma }=(y_1^{lt+1}+y_2^{t/p})$ holds by (2). 
Then, 
$A[g^{\pm 1}]\cap \kx =A[g]$ follows 
from (1) and Lemma~\ref{lem:intersection}. 
\end{proof}

Since $f=x_1x_3-x_2^t$ is irreducible, 
$\kx /f\kx $ is a $k$-domain 
of transcendence degree two. 
In what follows, 
$\ol{u}$ denotes the image of $u\in \kx $ in $\kx /f\kx $. 
Then,

\nd {\bf 15}$^\circ $ we have 
$\ol{q}_1=\ol{x}_1$ by (\ref{eq:rank3 q_1}), 
and $\ol{g}=\ol{x}_1^{mt-1}$ by (\ref{eq:rank3 g}) and 2$^\circ $.

\begin{lem}\label{lem:technical}
If $p\mid t$, then 
there exists $\xi \in q_1k[f,q_1]$ 
such that 
$q_2:=f^{-lp}(q-\xi )$ belongs to $\kx $, 
and $\ol{q}_1=\ol{x}_1$ and $\ol{q}_2$ 
are algebraically independent over $k$. 
\end{lem}

First, we prove Theorem~\ref{thm:rank3 invariant ring} (i) 
by assuming this lemma.

\begin{proof}[Proof of Theorem~{\rm \ref{thm:rank3 invariant ring} (i)}]
First, 
we construct $q_3$. 
Let $\xi $ and $q_2$ be as in Lemma~\ref{lem:technical}. 
Since $f^{lp}q_2=q-\xi $ and $\xi \in q_1k[f,q_1]$, 
we see that 
\begin{equation}\label{eq:rank3 lambda}
\lambda :=
q_1^{-1}(q^{t/p}-(f^{lp}q_2)^{t/p})=
q_1^{-1}(q^{t/p}-(q-\xi )^{t/p})\in k[f,q_1,q]. 
\end{equation}
Now, 
set $q_3:=f^{-lt}(g-\lambda )\in 
k[f^{\pm 1},g,q_1,q]$. 
Then, 
we have $g=f^{lt}q_3+\lambda $, 
and 
\begin{align*}
f^{lt+1}+q^{t/p}
\stackrel{\text{(\ref{eq:rank3 q_1})}}{=}
q_1g
=q_1(f^{lt}q_3+\lambda )
\stackrel{\text{(\ref{eq:rank3 lambda})}}{=}
f^{lt}q_1q_3+(q^{t/p}
-(f^{lp}q_2)^{t/p}). 
\end{align*}
This gives that $f=q_1q_3-q_2^{t/p}$. 
We show that $q_3$ is in $\kx $. 
By definition, $q_3$ is in $\kx [f^{-1}]$. 
Moreover, 
$q_1q_3$ is in $\kx $, 
since $q_1q_3=f+q_2^{t/p}$ 
and $f,q_2\in \kx $. 
By (\ref{eq:rank3 q_1}), 
$q_1$ is in $x_1+f\kx $, 
and so $\gcd (q_1,f)=1$. 
Hence, $q_3$ must lie in $\kx $.

Note that $A:=k[q_1,q_2,q_3]
\subset k[f^{\pm 1},g,q,q_1]$ 
by the definition of $q_2$ and $q_3$, 
and that $k[f,g,q,q_1]\subset k[f,q,q_1,q_3]
\subset A[f]=A$, 
since $g=f^{lt}q_3+\lambda \in k[f,q,q_1,q_3]$ 
by (\ref{eq:rank3 lambda}), 
$q=f^{lp}q_2+\xi \in k[f,q_1,q_2]$, 
and $f=q_1q_3-q_2^{t/p}\in A$. 
Hence, 
$k[f^{\pm 1},g,q,q_1]$ is equal to $A[f^{-1}]$. 
Thus, we get $\kx ^{\ep _h}=A[f^{-1}]\cap \kx $ 
by Proposition~\ref{prop:invariant ring first step}.

We show that 
$A[f^{-1}]\cap \kx =A$. 
For $\sigma :\ky \ni \nu (\y )\mapsto \nu (q_1,q_2,q_3)\in \kx $ 
and $\pi :\kx \to \kx /f\kx $, 
we set $\ol{\sigma }=\pi \circ \sigma $. 
Since $\ol{q}_1$ and $\ol{q}_2$ 
are algebraically independent over $k$ by Lemma~\ref{lem:technical}, 
we have $\trd _k\ol{\sigma }(\ky )=2$. 
Then, 
as in the proof of Proposition~\ref{prop:invariant ring first step}, 
the assertion holds by Lemma~\ref{lem:intersection}, 
because (1) $\sigma (y_1y_3-y_2^{t/p})
=q_1q_3-q_2^{t/p}=f\in fA[f]$ 
and 
(2) $y_1y_3-y_2^{t/p}$ is irreducible in $\ky $. 
\end{proof}

Next, 
we prove Lemma~\ref{lem:technical}. 
Set $R:=k[x_1,x_2,fx_3,f]$ 
and $M:=f^lx_1R+f^{2l}\kx $. 
Observe that $x_1,f,g,h,r,q,q_1\in R$ 
and $RM\subset M$. 
Hence, 
by (\ref{eq:rank3 q_1}), 
we see that 
$q_1\in x_1+f^lrR=x_1+f^l(x_1^m+f^lx_2)R\subset x_1+M$. 
Since $x_1M\subset M$, it follows that 

\nd {\bf 16}$^\circ $ 
$q_1^u\in x_1^u+M$ for all $u\ge 1$.

Since $h\in k[f,g]$, 
we can write $(gh)^{p-1}=\eta (g)$, 
where $\eta (T)\in k[f][T]$.

\begin{claim}\label{claim:1}
If $p\mid t$, then 
there exist $\lambda _2,\lambda _3\in x_1R$ such that 
$$
x_1^{mp}-q_1^{mp}\equiv f^{lp}\lambda _2,\ 
f^{l(p-1)}((gh)^{p-1}
x_1^m-\eta (q_1^{mt-1})q_1^m)
\equiv f^{lp}\lambda _3
\pmod{f^{l(p+1)}\kx }.
$$
\end{claim}
\begin{proof}
Since $x_1^m-q_1^m\in M$ by 16$^\circ $, 
we have 
$x_1^{mp}-q_1^{mp}=(x_1^m-q_1^m)^p\in (f^lx_1)^pR+f^{2lp}\kx $. 
Hence, there exists $\lambda _2$ as claimed. 
Since $p\mid t$, 
we have $g\in x_1^{mt-1}+M$ by 2$^\circ $. 
Hence, 
$g-q_1^{mt-1}$ lies in $M$ by 16$^\circ $. 
Since $\eta (T)$ is in $k[f][T]$, 
it follows that $\eta (g)-\eta (q_1^{mt-1})\in 
(g-q_1^{mt-1})k[f,g,q_1]\subset MR\subset M$. 
Thus, we get 
$$
\eta (g)x_1^m-\eta (q_1^{mt-1})q_1^m
=\eta (g)(x_1^m-q_1^m)
+(\eta (g)-\eta (q_1^{mt-1}))q_1^m\in M. 
$$
Since $\eta (g)=(gh)^{p-1}$, 
this shows that 
$f^{l(p-1)}((gh)^{p-1}x_1^m-\eta (q_1^{mt-1})q_1^m)$ 
belongs to 
$f^{l(p-1)}M=f^{lp}x_1R+f^{l(p+1)}\kx $. 
Therefore, 
there exists $\lambda _3$ as claimed. 
\end{proof}

\begin{proof}[Proof of Lemma~$\ref{lem:technical}$]
We show that the assertion holds for 
\begin{equation}\label{eq:rank3 xi}
\xi :=\left\{ \begin{array}{ll}
q_1^{mp}&\text{ (1) if }h^{p-1}\in f^lk[f,g] \\
q_1^{mp}-f^{l(p-1)}\eta (q_1^{mt-1})q_1^m&\text{ (2) otherwise. }
\end{array}
\right.
\end{equation}
First, observe that 
$q=f^{lp}(x_2^p-(gh)^{p-1}x_2)+x_1^{mp}-f^{l(p-1)}(gh)^{p-1}x_1^m$ 
by (\ref{eq:rank3 q}).

(1) 
We can write $f^{l(p-1)}(gh)^{p-1}x_1^m=f^{lp}\lambda _1$, 
where $\lambda _1\in x_1^mk[f,g]\subset x_1R$. 
Let $\lambda _2\in x_1R$ be as in Claim~\ref{claim:1}. 
Then, 
$x_1^{mp}-q_1^{mp}$ is in $f^{lp}\lambda _2+f^{l(p+1)}\kx $. 
Hence, 
$q_2=f^{-lp}(q-\xi )=f^{-lp}(q-q_1^{mp})$ is in 
$x_2^p-(gh)^{p-1}x_2+\lambda _2-\lambda _1+f^l\kx \subset \kx $. 
This also shows that 
$\ol{q}_2
=\ol{x_2^p+\lambda }$ 
for some $\lambda \in x_1k[x_1,x_2]$, 
since $\ol{g}=\ol{x}_1^{mt-1}$ by 15$^\circ $, 
$h\in R$, 
$\lambda _1,\lambda _2\in x_1R$, 
and the image of $R$ in $\kx /f\kx $ is $k[\ol{x}_1,\ol{x}_2]$. 
Since $x_2^p+\lambda $ is in $k[x_1,x_2]\sm k[x_1]$, 
this implies that $\ol{x}_1$ and $\ol{q}_2$ are algebraically 
independent over $k$, 
for otherwise $\nu (x_1,x_2^p+\lambda )\in f\kx $ for some 
$\nu \in k[x_1,x_2]\sm \zs $, 
which is absurd.

(2) 
We have $q_2=f^{-lp}(q-\xi )\in 
x_2^p-(gh)^{p-1}x_2+\lambda _2-\lambda _3+f^l\kx $, 
where 
$\lambda _2,\lambda _3\in x_1R$ are as in Claim~\ref{claim:1}. 
Then, 
the assertion is verified as in (1). 
\end{proof}

\begin{proof}[Proof of Theorem~{\rm \ref{thm:rank3 invariant ring} (ii)}]
First, we remark that 
$\psi (x_2^p)=\psi (f^{-lp}(r^p-x_1^{mp}))
=f^{-lp}(q-q_1^{mp})$ 
by 1$^\circ $, 
since $\psi (x_1)=q_1$ as mentioned.

If $h^{p-1}\in f^lk[f,g]$, 
then by (\ref{eq:rank3 xi}), 
we may take $\xi =q_1^{mp}$ 
in the proof of (i). 
Then, 
we have 
$q_2=f^{-lp}(q-q_1^{mp})$, 
which equals $\psi (x_2^p)$ as remarked. 
Hence, we get 
$f=\psi (f)=\psi (x_1x_3-x_2^{(t/p)p})
=q_1\psi (x_3)-q_2^{t/p}$. 
Since $q_1q_3-q_2^{t/p}=f$, 
this gives that $\psi (x_3)=q_3$. 
Thus, we get 
$\psi (k[x_1,x_2^p,x_3])=k[q_1,q_2,q_3]=\kx ^{\ep _h}$ by (i).

For the converse, 
assume that $h^{p-1}\not\in f^lk[f,g]$. 
It suffices to prove $\psi (x_2^p)\not\in \kx $. 
Set $\lambda _1:=f^{-l}(gh)^{p-1}x_1^m$, 
and let $\lambda _2$ be as in Claim~\ref{claim:1}. 
Then, 
$f^{-lp}(q-q_1^{mp})$ is in 
$x_2^p-(gh)^{p-1}x_2+\lambda _2-\lambda _1+f^l\kx $ 
as shown in (1) of the proof of Lemma~\ref{lem:technical}. 
Moreover, 
$x_2^p-(gh)^{p-1}x_2+\lambda _2$ is in $\kx $. 
Hence, we get $\psi (x_2^p)\in -\lambda _1+\kx $ by the remark. 
Now, we claim that $f^{-l}h^{p-1}\not\in \kx $, 
for otherwise 
$f^{-l}h^{p-1}\in k[f^{\pm },g]\cap \kx =k[f,g]$ 
by (ii) of 
the proof of Theorem~\ref{thm:k[f,g]}, 
and so $h^{p-1}\in f^lk[f,g]$, a contradiction. 
Since $\gcd (f,g)=1$ by 4$^\circ $, 
it follows that $\lambda _1\not\in \kx $. 
This proves that $\psi (x_2^p)\not\in \kx $. 
\end{proof}

\begin{rem}\label{rem:t=p, f q_2}\rm 
If $t=p$, then the following statements hold. 

\nd {\rm (i)} 
Since $q_1q_3-f=q_2^{t/p}=q_2$, 
we have $\kx ^{\ep _h}=k[q_1,q_2,q_3]=k[f,q_1,q_3]$. 

\nd {\rm (ii)} 
$\lambda $ in (\ref{eq:rank3 lambda}) is equal to $\xi q_1^{-1}$. 
Hence, we have $q_3=f^{-lt}(g-\xi q_1^{-1})$. 
\end{rem}

\begin{example}\label{example:simple example}\rm 
Let $(l,t)=(1,p)$, 
and $m=2$ if $p=2$, 
and $m=1$ if $p\ge 3$. 
Put $s:=mt-1$, 
i.e., 
$s=3$ if $p=2$, 
and $s=p-1$ if $p\ge 3$. 
For $h:=f$, 
we set $\phi :=\ep _h$. 
Then, 
we have $f=x_1x_3-x_2^p$, 
$g=f^px_3+x_1^s$ by 2$^\circ $, 
\begin{equation}\label{eq:simple example}
\begin{gathered}
\phi (x_1)=x_1+f^{2p}g^{p-1},\quad 
\phi (x_2)=x_2+fg-
\left\{ \begin{array}{cc}
f^7g^2 & \text{if}\ p=2 \\
f^{2p-1}g^{p-1}&\text{if}\ p\ge 3
\end{array}\right.
\\
\phi (x_3)=x_3-f^{-p}\bigl( (x_1+f^{2p}g^{p-1})^{s}-x_1^{s} \bigr)
\end{gathered}
\end{equation}
by Lemma~\ref{lem:rank3 J} (iii), 
$\gamma (\kx ^{\phi })=0$ by Corollary~\ref{cor:rank3} (2), 
and $\pl (\phi )=fg\kx ^{\phi }$ 
by Theorem~\ref{thm:plinth rank 3} (ii). 
Since $h^{p-1}=f^{p-1}$ is in $fk[f,g]$, 
we may take $\xi =q_1^{s+1}$ by (\ref{eq:rank3 xi}), 
where $q_1:=g^{-1}(f^{p+1}+q)$. 
Then, 
we have $q_3:=f^{-p}(g-q_1^{s})$ and 
$\kx ^{\phi }=k[f,q_1,q_3]$ by Remark~\ref{rem:t=p, f q_2}. 
We also have $\psi (k[x_1,x_2^p,x_3])=\kx ^{\phi }$ 
by Theorem~\ref{thm:rank3 invariant ring} (ii). 
\end{example}

\subsection{Proof of Theorem~\ref{thm:rank3 invariant ring2}}
\label{sect:rank3 invariant proof2}

The goal of \S \ref{sect:rank3 invariant proof2} 
is to prove Theorem~\ref{thm:rank3 invariant ring2}. 
Throughout, 
we assume that $p\nmid t$, $p\nmid mt-1$ 
and $h^{p-1}\in f^{l+1}g^2k[f,g]\sm \zs $. 
First, we note the following:

\nd {\bf 17}$^\circ $ 
$q=r^p-(f^lgh)^{p-1}r$ belongs to 
$r^p+f^{lp+1}g^{p+1}\kx \subset r^p+f^{lp+1}g\kx $.

\nd {\bf 18}$^\circ $ 
By 1$^\circ $ and 2$^\circ $, 
we can write 
$x_1^p=g^{-p}\eta _1(r^p)$ 
and $x_3^p=f^{-ltp}\eta _3(x_1^p,x_2^p)$, 
where $\eta _1(T)\in k[f,T]$ 
and $\eta _3(T,U)\in k[f,g,T,U]$. 
We also have $x_2^p=f^{-lp}(r^p-x_1^{mp})$.

\begin{lem}\label{lem:rank3 invariant ring2}
\nd{\rm (i)} 
$p_i:=\psi (x_i^p)$ belongs to $x_i^p+fg\kx $ for $i=1,2,3$. 

\nd{\rm (ii)} 
We have $\kx ^{\ep _h}=C[\bp ]\cap \kx $, 
where $\bp :=\{ p_1,p_2,p_3\} $. 

\end{lem}
\begin{proof}
(i) 
We have $\psi (x_1^p)=g^{-p}\eta _1(q)$ by 18$^\circ $, 
and 
$\eta _1(q)\in \eta _1(r^p)+f^{lp+1}g^{p+1}\kx $ by 17$^\circ $. 
Hence, we get 
$\psi (x_1^p)\in g^{-p}\eta _1(r^p)+f^{lp+1}g\kx 
=x_1^p+f^{lp+1}g\kx \subset x_1^p+fg\kx $ 
by 18$^\circ $. 
From this and 17$^\circ $, 
we have 
$q-\psi (x_1^p)^m\in r^p-x_1^{mp}+f^{lp+1}g\kx $. 
Hence, 
we get 
$\psi (x_2^p)=f^{-lp}(q-\psi (x_1^p)^m)
\in f^{-lp}(r^p-x_1^{mp})+fg\kx =x_2^p+fg\kx $ 
by 18$^\circ $.

Since $\psi (x_1^p),\psi (x_2^p)\in \kx $, 
we have 
$\psi (x_3^p)=f^{-ltp}\eta _3(\psi (x_1^p),\psi (x_2^p))\in 
f^{-ltp}\kx $ by 18$^\circ $, 
and $\psi (x_1^p)\psi (x_3^p)
=\psi (x_1^px_3^p)
=\psi (x_2^{tp}+f^p)
=\psi (x_2^p)^t+f^p\in \kx $. 
We also have $\gcd (f,\psi (x_1^p))=1$, 
since $\psi (x_1^p)\in x_1^p+fg\kx $. 
Thus, 
$\psi (x_3^p)$ belongs to $\kx $.

In $\kx /g\kx $, 
the equation 
$\psi (x_1^p)\psi (x_3^p)-\psi (x_2^p)^t
=\psi (f^p)=f^p=x_1^px_2^p-x_2^{tp}$ 
yields $\ol{x_1^p(\psi (x_3^p)-x_3^p)}=\ol{0}$, 
since $\ol{\psi (x_i^p)}=\ol{x_i^p}$ for $i=1,2$ as shown above. 
Since $\ol{x}_1\ne \ol{0}$ and 
$\kx /g\kx $ is a domain by Lemma~\ref{lem:K_g}, 
it follows that $\ol{\psi (x_3^p)-x_3^p}=\ol{0}$, i.e., 
$\psi (x_3^p)-x_3^p\in g\kx $. 
We can prove $\psi (x_3^p)-x_3^p\in f\kx $ similarly. 
Therefore, 
$\psi (x_3^p)-x_3^p$ belongs to 
$f\kx \cap g\kx =fg\kx $ by 4$^\circ $.

(ii) 
From $p_2=f^{-lp}(q-p_1^m)$, 
we see that $q$ is in $C[\bp ]$. 
Since $C[\bp ]\subset C[q]$, 
we get $C[\bp ]=C[q]$. 
Then, 
the assertion follows from 
(\ref{eq:C[r]^{ep_h}}). 
\end{proof}

For $f_1,f_2\in \kx $, 
we define the {\it Jacobian derivation} 
$D_{(f_1,f_2)}:\kx \to \kx $ by 
$D_{(f_1,f_2)}(f_3):=\det (\partial f_i/\partial x_j)_{i,j}$ 
for each $f_3\in \kx $.

\begin{lem}\label{lem:derivation}
We have $D_{(x_2,f)}(g)\not\in f\kx $ 
and $D_{(x_1,g)}(f)\not\in g\kx $. 
\end{lem}
\begin{proof}
Set $D:=D_{(x_2,f)}$. 
Then, 
$\ker D$ contains $k[x_2,f]$. 
We also have 
$D(x_1)=D_{(x_2,f)}(x_1)=-D_{(x_2,x_1)}(f)=D_{(x_1,x_2)}(f)=\partial f/\partial x_3=x_1$. 
Hence, 
we get 
$D(r)=D(f^lx_2+x_1^m)=mx_1^{m-1}D(x_1)=mx_1^m$. 
Since $x_1g=f^{lt+1}+r^t$ by (\ref{eq:rank3 g}), 
it follows that 
$x_1D(g)+x_1g=
x_1D(g)+D(x_1)g=D(x_1g)=D(f^{lt+1}+r^t)=tr^{t-1}D(r)=mtr^{t-1}x_1^m$. 
This gives that 
$D(g)=mtr^{t-1}x_1^{m-1}-g$. 
Now, recall that $r\in x_1^m+f\kx $, 
and $g\in x_1^{mt-1}+f\kx $ by 2$^\circ $. 
Hence, 
$D(g)$ is in $(mt-1)x_1^{mt-1}+f\kx $. 
Since $p\nmid mt-1$ by assumption, 
we know that $D(g)\not\in f\kx $.

Next, 
set $E:=D_{(x_1,f)}$. 
Since $D_{(x_1,g)}(f)=-E(g)$, 
we show that $E(g)\not\in g\kx $. 
Since $k[x_1,f]\subset \ker E$ and $E(x_2)=-D_{(x_1,x_2)}(f)=-x_1$, 
we have 
$E(r)=E(f^lx_2+x_1^m)=f^lE(x_2)=-f^lx_1$, 
and 
$x_1E(g)=E(x_1g)=
E(f^{lt+1}+r^t)=tr^{t-1}E(r)
=-tr^{t-1}f^lx_1$. 
Since $p\nmid t$ by assumption, 
the last equation implies 
$E(g)\not\in g\kx $. 
\end{proof}

Let $\ol{k}$ be an algebraic closure of $k$, 
and $F:\ol{k}\ni a\mapsto a^p\in \ol{k}$. 
Then, 
$\varphi :=F\otimes \id _{\ol{k}[\x ]}$ 
is an automorphism of $\ol{k}\otimes _{\ol{k}}\ol{k}[\x ]
=\ol{k}[\x ]$ over $\bF _p[\x ]$. 
We define 
$$
\tau :k[\x ,y,z]
\ni u(x_1,x_2,x_3,y,z)\mapsto u(x_1^p,x_2^p,x_3^p,f,g)\in \kx .
$$

\begin{rem}\label{rem:Frob}\rm 
(i) $\tau (u)=u^p$ holds for all $u\in \bF _p[\x ]$, 
and hence for $u=f,g$. 

\nd (ii) For each $\lambda \in \kx $, 
we have $\tau (\lambda )=\varphi ^{-1}(\lambda )^p$. 
Hence, if 
$u\in \bF _p[\x ]$ is irreducible in $\ol{k}[\x ]$, 
then the following implication holds: 
$\tau (\lambda )\in u\kx \Rightarrow 
 \varphi ^{-1}(\lambda )\in u\ol{k}[\x ]
\Rightarrow 
\lambda =\varphi (\varphi ^{-1}(\lambda ))\in 
\varphi (u\ol{k}[\x ])=u\ol{k}[\x ]
\Rightarrow \lambda /u\in \ol{k}[\x ]\cap \kxr =\kx 
\Rightarrow \lambda \in u\kx $. 
\end{rem}

\begin{lem}\label{lem:derivation2}

\nd{\rm (i)} 
If $\lambda \in k[\x ,y]$ satisfies 
$\deg _y\lambda <p$ and 
$\tau (\lambda )\in g\kx $, 
then $\lambda $ belongs to $ gk[\x ,y]$.

\nd{\rm (ii)} 
If $\lambda \in k[\x ,z]$ satisfies 
$\deg _z\lambda <p$ and 
$\tau (\lambda )\in f\kx $, 
then $\lambda $ belongs to $fk[\x ,z]$. 

\end{lem}
\begin{proof}
(i) 
Recall that $g$ is in $\bF _p[\x]$, 
and is irreducible in $\ol{k}[\x ]$, 
since Lemma~\ref{lem:K_g} holds for any $k$. 
Now, suppose that 
there exists $\lambda \in k[\x ,y]\sm gk[\x ,y]$ with 
$\deg _y\lambda <p$ and $\tau (\lambda )\in g\kx $. 
Write $\lambda =\sum _{i=0}^{p-1}\lambda _iy^i$, 
where $\lambda _i\in \kx $. 
For $i$ with $\lambda _i\in g\kx $, 
we have $\tau (\lambda _i)\in \tau (g\kx )\subset g^p\kx $ 
by Remark~\ref{rem:Frob} (i). 
Hence, 
by subtracting $\lambda _iy^i$ from $\lambda $ for such $i$, 
we may assume that 
$\lambda =\sum _{i=0}^s\lambda _iy^i$ 
and $\lambda _s\not\in g\kx $ 
for some $0\le s<p$. 
Choose $\lambda $ with least $s$. 
Then, we claim that $s\ge 1$, 
for otherwise 
$\tau (\lambda _0)=\tau (\lambda )\in g\kx $ 
and $\lambda _0\not\in g\kx $, 
contradicting Remark~\ref{rem:Frob} (ii).

Set $\lambda ':=\sum _{i=0}^si\lambda _iy^{i-1}$. 
Since $1\le s<p$, 
we have $\deg _y\lambda '=s-1$ and $s\lambda _s\not\in g\kx $. 
Hence, $\tau (\lambda ')$ is not in $g\kx $ 
by the minimality of $s$. 
Now, set $D:=D_{(x_1,g)}$. 
Then, 
since $\tau (\lambda )\in g\kx $ and $D(g)=0$, 
we have $D(\tau (\lambda ))\in D(g\kx )\subset g\kx $. 
On the other hand, 
since $D$ kills $\tau (\lambda _i)
=\lambda _i(x_1^p,x_2^p,x_3^p)$, 
we have 
$D(\tau (\lambda ))
=D(\sum _{i=0}^s\tau (\lambda _i)f^i)
=\sum _{i=0}^s\tau (\lambda _i)D(f^i)
=\sum _{i=0}^si\tau (\lambda _i)f^{i-1}D(f)
=\tau (\lambda ')D(f)$. 
Since $D(f)\not\in g\kx $ by Lemma~\ref{lem:derivation}, 
it follows that $\tau (\lambda ')\in g\kx $, 
a contradiction.

We can prove (ii) as in (i) 
using $D_{(x_2,f)}$ instead of $D_{(x_1,g)}$. 
\end{proof}

\begin{proof}[Proof of Theorem~{\rm \ref{thm:rank3 invariant ring2}}]
(i) 
By Lemma~\ref{lem:rank3 invariant ring2} (ii) 
and Remark~\ref{rem:intersection}, 
it suffices to show that 
(a) $k[\bp ,f][g^{\pm 1}]\cap \kx =k[\bp ,f][g]$, and 
(b) $k[\bp ,g][f^{\pm 1}]\cap \kx =k[\bp ,g][f]$. 
We only prove (a), 
since the proof of (b) is similar.

For $\sigma :k[\x ,y]\ni u(\x ,y)\mapsto u(\bp ,f)\in \kx $ 
and $\pi :\kx \to \kx /g\kx $, 
we set $\ol{\sigma }:=\pi \circ \sigma $. 
By Lemma~\ref{lem:rank3 invariant ring2} (i), 
$\ol{\sigma }(x_i)=\pi (p_i)$ 
equals $\pi (x_i^p)=\pi (\tau (x_i))$ 
for each $i$. 
Moreover, 
we have $\sigma (y)=f=\tau (y)$. 
Hence, we get 

\nd {\bf 19}$^\circ $ 
$\ol{\sigma }=\pi \circ \tau |_{k[\x ,y]}$.

\nd By definition, 
$\sigma (x_i)=p_i=\psi (x_i^p)=\psi (\tau (x_i))$ 
holds for each $i$. 
Hence, 
for $u=f,g$, 
we have $\sigma (u)=\psi (\tau (u))=\psi (u^p)=u^p$ 
by Remark~\ref{rem:Frob} (i). 
Thus, 
we know that

\nd {\bf 20}$^\circ $ 
$\sigma (g)=g^p$ 
and 
$\sigma (y^p-f)=f^p-f^p=0$ 
belong to 
$gk[\bp ,f][g]$.

Now, 
we show that $\ker \ol{\sigma }=(g,y^p-f)$. 
Then, 
(a) follows by Lemma~\ref{lem:intersection} and 20$^\circ $. 
By 20$^\circ $, 
we have $g,y^p-f\in \ker \ol{\sigma }$. 
To show 
$\ker \ol{\sigma }\subset (g,y^p-f)$, 
pick any $\eta \in \ker \ol{\sigma }$. 
Write $\eta =(y^p-f)\eta _0+\eta _1$, 
where $\eta _0,\eta _1\in k[\x ,y]$ 
with $\deg _y\eta _1<p$. 
Then, 
$\eta _1$ is in $\ker \ol{\sigma }$, 
since $\eta ,y^p-f\in \ker \ol{\sigma }$. 
By 19$^\circ $, 
this implies that $\tau (\eta _1)\in \ker \pi =g\kx $. 
Hence, we get $\eta _1\in gk[\x ,y]$ by Lemma~\ref{lem:derivation2} (i). 
Therefore, $\eta $ belongs to $(g,y^p-f)$.

(ii) 
By (i), 
we have 
$\kx ^{\ep _h}\simeq k[\x ^p,f,g]
=\tau (k[\x ,y,z])\simeq k[\x ,y,z]/\ker \tau $, 
where $\x ^p:=\{ x_1^p,x_2^p,x_3^p\} $. 
So, we show that $\ker \tau =
(y^p-f,z^p-g)$. 
Since ``$\supset $" is clear, 
we only check ``$\subset $." 
First, 
we claim that 
$(f^ig^jr^m)_{0\le i,j,m<p}$ is a $k(\x ^p)$-basis of $\kxr $, 
since $f^p,g^p,r^p\in k(\x ^p)$, 
$[\kxr :k(\x ^p)]=p^3$ 
and 
$\kxr =k(f,g,r)=k(\x ^p,f,g,r)$ by 1$^\circ $. 
Now, 
pick any $\eta \in \ker \tau $. 
Write 
$\eta =(y^p-f)\eta _1+(z^p-g)\eta _2+\sum _{i,j=0}^{p-1}\eta _{i,j}y^iz^j$, 
where $\eta _1,\eta _2\in k[\x ,y,z]$ 
and $\eta _{i,j}\in \kx $. 
Then, 
we have 
$\sum _{i,j=0}^{p-1}\tau (\eta _{i,j})f^ig^j=\tau (\eta )=0$. 
Hence, by the claim, 
$\tau (\eta _{i,j})=\eta _{i,j}(x_1^p,x_2^p,x_3^p)$ 
must be zero for all $i$, $j$. 
This implies that $\eta _{i,j}=0$ for all $i$, $j$. 
Therefore, 
$\eta $ belongs to $(y^p-f,z^p-g)$.

Set 
$(f_1,f_2,x_4,x_5):=(y^p-f,z^p-g,y,z)$. 
If $k[\x ,x_4,x_5]/(f_1,f_2)\simeq \kx $, 
then the affine variety $f_1=f_2=0$ 
in $\A _k^5$ 
is isomorphic to $\A _k^3$, 
and hence smooth. 
However, 
this affine variety has a singular point at the origin, 
because 
$((\partial f_i/\partial x_j)(0))_{i,j}$ 
is a zero matrix, 
and of rank less than two. 
Hence, 
$\kx ^{\ep _h}$ is not isomorphic to $\kx $. 
\end{proof}

\section{Nagata type automorphisms}\label{sect:Nagata}
\setcounter{equation}{0}

In this section, 
we study the Nagata type automorphisms. 
In \S \ref{sect:Nagata construction}, 
we construct the automorphism. 
In \S \ref{sect:Nagata pl}, 
we study the plinth ideals. 
In \S \ref{sect:Dedekind}, 
we prove a theorem 
used to construct a generator of the invariant ring. 
The invariant rings are studied 
in \S \ref{sect:Nagata invariant ring1} 
and \S \ref{sect:Nagata invariant ring2}. 
Theorem~\ref{thm:Nagata Main} 
easily 
follows from 
Theorems~\ref{thm:pl princ Nagata}, 
\ref{thm:Nagata2}, 
\ref{thm:Nagata relation1} 
and \ref{thm:Nagata relation2}.

\subsection{Construction}\label{sect:Nagata construction}

Let $R$ be a UFD with $p:=\ch R>0$, 
and $R[x,y]$ the polynomial ring in two variables over $R$. 
For $g,h\in R[x,y]$ and $c\in R$, 
we write $g\equiv _ch$ if $g-h\in cR[x,y]$. 
For $c\in R\sm \zs $, 
we write $R_c:=R[c^{-1}]$. 
We remark that 
\begin{equation}\label{eq:R_c}
R_{c_1^{l_1}\cdots c_t^{l_t}}=
R_{c_1\cdots c_t}=R[c_1^{-1},\ldots ,c_t^{-1}]
\quad (\forall 
c_1,\ldots ,c_t\in R\sm \zs ,\ 
l_1,\ldots ,l_t\ge 1).
\end{equation}

Now, 
we fix $a\in R\sm \zs $ and $\theta (y)\in yR[y]\sm \zs $, 
and define 
$f:=ax+\theta (y)$. 
Then, 
$R_a[x,y]=R_a[f][y]$ 
is the polynomial ring in $y$ over $R_a[f]$. 
Hence, by Example~\ref{example:action on A[x]}, 
$\widetilde{\ep }:R_a[f][y]\ni u(y)\mapsto u(y+aT)\in R_a[f][y][T]$ 
is a $\Ga $-action on $R_a[x,y]$ 
with $R_a[x,y]^{\widetilde{\ep }}=R_a[f]$. 
Since $ax+\theta (y)=f=\widetilde{\ep }(f)
=a\widetilde{\ep }(x)+\theta (y+aT)$, 
we have 
\begin{equation}\label{eq:Nagata action}
\widetilde{\ep }(x)
=x+a^{-1}(\theta (y)-\theta (y+aT))
=x-\theta '(y)T
-aT^2\sum _{i\ge 2}\theta _i(y)(aT)^{i-2}
\end{equation}
(cf.~Notation~\ref{notation:taylor}). 
Hence, 
$\widetilde{\ep }(x)$ is in $R[x,y][T]$. 
Since $\widetilde{\ep }(y)=y+aT$ 
is in $R[x,y][T]$ by definition, 
$\widetilde{\ep }$ restricts to a $\Ga $-action $\ep $ on $R[x,y]$ 
with $R[x,y]^{\ep }=R_a[f]\cap R[x,y]$.

We fix 
$0\ne F\in R[f]\subset R[x,y]^{\ep }$, 
and define 
\begin{equation}\label{eq:Nagata q}
\phi :=\ep _F\in \Aut _RR[x,y]
\quad\text{and}\quad 
q:=y^p-(aF)^{p-1}y.
\end{equation}
This $\phi $ is the same as 
$\psi $ in Theorem~\ref{thm:Nagata Main}, 
namely, 
the Nagata type automorphism. 
The purpose of Section~\ref{sect:Nagata} is to 
study the structures of $R[x,y]^{\phi }$ and $\pl (\phi )$. 
Note that $\phi $ is the restriction of 
$\widetilde{\phi }:R_a[f][y]\ni u(y)\mapsto u(y+aF)\in R_a[f][y]$ to $R[x,y]$. 
Hence, 
by Lemma~\ref{lem:1var}, 
we have 
\begin{equation}\label{eq:Nagata invariant}
R[x,y]^{\phi }=R_a[f][y]^{\widetilde{\phi }}\cap R[x,y]=R_a[f,q]\cap R[x,y]. 
\end{equation}

\begin{example}\label{example:Nagata}\rm
Let $R=k[z]$ be the polynomial ring in one variable over a field $k$. 
If $a=z$, $\theta (y)=y^2$ and $F=f$, 
then we have $\phi (x)=x-2yf-zf^2$ and $\phi (y)=y+zf$. 
This $\phi $ is the famous automorphism of Nagata~\cite{Nagata}. 
\end{example}

Write $\theta (y)=\sum _{i\ge 1}s_iy^i$, 
where $s_i\in R$. 
Then, we have $\theta '(y)=\sum _{p\nmid i}is_iy^{i-1}$. 
We define  
\begin{equation}\label{eq:Nagata notation}
I:=aR[x,y]+\theta '(y)R[x,y], \ 
d:=\gcd (a,\theta '(y)), \  
b:=ad^{-1},\ 
\rho (y):=\sum _{p\nmid i}t_iy^i,
\end{equation}
where $t_i:=s_id^{-1}$ for each $i$ with $p\nmid i$.

\begin{rem}\label{rem:NGT}\rm 
$\theta '(y)=d\rho '(y)$ and 
$\gcd (b,\rho (y))=\gcd (b,\rho '(y))=1$. 
\end{rem}

\begin{rem}\label{rem:principality of I}\rm 
We have $I\subset dR[x,y]$. 
Moreover, 
the following A through F are equivalent 
(see 
\cite[\S 1, Exercise 2]{AM} for D $\Leftrightarrow $ E, and 
\cite[\S 0.3]{Nagata} for E $\Leftrightarrow $ F).

\nd {\bf A}. 
$I$ is a principal ideal of $R[x,y]$.

\nd {\bf B}. 
$d$ belongs to $I$, that is, 
$I=dR[x,y]$.

\nd {\bf C}. 
There exist 
$\zeta _1,\zeta _2\in R[y]$ such that 
$a\zeta _1+\theta '(y)\zeta _2=d$, i.e., 
$b\zeta _1+d^{-1}\theta '(y)\zeta _2=1$.

\nd {\bf D}. 
The image of 
$d^{-1}\theta '(y)=\sum _{p\nmid i}it_iy^{i-1}$ in $(R/bR)[y]$ 
is a unit of $(R/bR)[y]$.

\nd {\bf E}. 
We have $\ol{t}_1\in (R/bR)^*$, i.e., $1\in t_1R+bR$. 
For each $i\ge 2$ with $p\nmid i$, 
we have $\ol{t}_i\in \nil (R/bR)$, 
i.e., $t_i\in \sqrt{bR}$. 
Here, $\nil (S)$ denotes the nilradical of a ring $S$.

\nd {\bf F}. 
$(R/bR)[\ol{\rho }(y)]=(R/bR)[y]$, 
i.e., 
$y\equiv _b\nu (\rho (y))$ for some $\nu (y)\in R[y]$.

We note that, 
if $b=1$, 
i.e., 
$R/bR=\zs $, 
then D, E and F are trivial. 
\end{rem}

Now, observe that 
$\phi (y)=y+dbF$ and $f-\theta (y)=dbx$. 
Set $\theta ^*(y):=\sum _{i\ge 1}s_{pi}y^i$. 
Then, we have 
$\theta (y)=\theta ^*(y^p)+d\rho (y)$. 
Hence, 
we can apply Remark~\ref{rem:q_1} (ii) with 	
$$
(a,b,c,S,\xi (T),w,\xi ^*(T),\widehat{\xi }(T))
=(d,bF,y,R[f],f-\theta (T),bx,f-\theta ^*(T),\rho (T)). 
$$ 
Thus, 
by (\ref{eq:q_1^*}), we know that $R[x,y]^{\phi }$ contains 
\begin{equation}\label{eq:q_1 Nagata}
q_1:=d^{-1}(f-\theta ^*(q))
=d^{-1}\xi ^*(q)
\in bx+\rho (y)+d^{p-2}(bF)^{p-1}yR[f,aF,y]. 
\end{equation}
Note that $R_a[q_1,q]=R_a[f,q]$, 
since $d^{-1}=b/a\in R_a$. 
Hence, 
it follows from (\ref{eq:Nagata invariant}) 
that $R[x,y]^{\phi }=R_a[q_1,q]\cap R[x,y]$. 
In fact, 
the following theorem holds.

\begin{thm}\label{thm:Nagata1}
We have $R[x,y]^{\phi }=R_b[q_1,q]\cap R[x,y]$. 
\end{thm}
\begin{proof}
It suffices to verify 
$R_a[q_1,q]\cap R[x,y]=R_b[q_1,q]\cap R[x,y]$. 
Let $d_0$ be the product of 
all prime factors of $d$ 
not dividing $b$. 
Since $a=bd$, 
we have $R_a=R_{bd_0}$ by (\ref{eq:R_c}). 
Hence, 
by Remark~\ref{rem:intersection}, 
we are reduced to proving that 
$R_{d_0}[q_1,q]\cap R[x,y]=R[q_1,q]$. 
Due to Remark~\ref{rem:indep2}, 
it suffices to show that 
$\mu (\ol{q}_1,\ol{q})\ne 0$ 
for all $\mu (x,y)\in (R/d_0R)[x,y]\sm \zs $. 
Here, $\ol{h}$ denotes the image of $h$ in $(R/d_0R)[x,y]$ 
for $h\in R[x,y]$. 
Note that $\gcd (b,d_0)=1$ by the definition of $d_0$. 
Hence, 
$\ol{b}$ is not a zero-divisor of $R/d_0R$. 
Since $d_0\mid a$, 
we have $\ol{f}=\ol{\theta }(y)$. 
Hence, 
$\ol{q}_1=\ol{b}x+\ol{\eta }(y)$ for some $\eta (y)\in R[y]$ 
by (\ref{eq:q_1 Nagata}). 
We also have $\ol{q}=y^p$ 
by (\ref{eq:Nagata q}). 
Now, 
pick any $\mu (x,y)=\sum _{i=0}^l\mu _i(y)x^i\in (R/d_0R)[x,y]\sm \zs $, 
where $\mu _i(y)\in (R/d_0R)[y]$, 
$l\ge 0$ and $\mu _l(y)\ne 0$. 
Then, since $\mu _l(y^p)\ol{b^l}\ne 0$, 
we get $\mu (\ol{q}_1,\ol{q})=\mu (\ol{b}x+\ol{\eta }(y),y^p)
=\mu _l(y^p)\ol{b^l}x^l+\cdots \ne 0$. 
\end{proof}

\begin{rem}\label{rem:Nagata psi}\rm 
(i) 
If $\theta '(y)\in aR[y]$, i.e., 
$b=1$, 
then we have 
$R[x,y]^{\phi }=R[q_1,q]$ by Theorem~\ref{thm:Nagata1}.

\nd (ii) 
We define an isomorphism 
$\psi :R_a[f][y^p]\ni u(y^p)
\mapsto u(q)\in R_a[f,q]=R_a[x,y]^{\widetilde{\phi }}$. 
If $\theta '(y)=0$, i.e., 
$\theta (y)\in R[y^p]$, 
then $x=a^{-1}(f-\theta ^*(y^p))$ lies in $R_a[f][y^p]$, 
and $\psi (x)=q_1$ by (\ref{eq:q_1 Nagata}). 
Hence, 
we have $\psi (R[x,y^p])=R[q_1,q]=R[x,y]^{\phi }$ by (i). 
This can be considered as an analogue of 
Theorems~\ref{thm:rank3 invariant ring} (ii) 
and \ref{thm:rank3 invariant ring2}. 
\end{rem}

\begin{example}
If $p=2$ in Example~\ref{example:Nagata}, 
then we have $\theta '(y)=2y=0$ and $\theta ^*(y)=y$. 
Hence, 
we get $R[x,y]^{\phi }=R[q_1,q]=\psi (R[x,y^p])$ 
by Remark~\ref{rem:Nagata psi} (ii), 
where $q=y^2-zfy$ and $q_1=z^{-1}(f-q)=x+fy$. 
\end{example}

\subsection{Plinth ideals}\label{sect:Nagata pl}

We set $\delta :=\phi -\id $. 
Recall that $I=aR[x,y]+\theta '(y)R[x,y]$. 
From (\ref{eq:Nagata action}), 
we see that $\delta (x)=\widetilde{\ep }(x)|_{T=F}-x\in FI$. 
Since $\delta (y)=aF$ is in $FI$, 
we know by \S \ref{sect:plinth} (5) 
that 
\begin{equation}\label{eq:delta ideal}
\pl (\phi )
=\delta (R[x,y])\cap R[x,y]^{\phi }
\subset \delta (R[x,y])\subset FI.
\end{equation}

The goal of \S \ref{sect:Nagata pl} is to 
prove the following theorem. 

\begin{thm}\label{thm:pl princ Nagata}
The following are equivalent: 

\nd {\rm a.} $I$ is a principal ideal of $R[x,y]$. 

\nd {\rm b.} $I=dR[x,y]$. 

\nd {\rm c.} $\pl (\phi )=dFR[x,y]^{\phi }$. 

\nd {\rm d.} $\pl (\phi )$ is a principal ideal of $R[x,y]^{\phi }$. 

\end{thm}

We have ``a $\Leftrightarrow $ b" by Remark~\ref{rem:principality of I} B, 
and ``c $\Rightarrow $ d" is clear. 
In the rest of \S \ref{sect:Nagata pl}, 
we prove 
``b $\Rightarrow $ c" and ``d $\Rightarrow $ b". 
Our tools are 
Lemma~\ref{lem:pl principal} and the following lemma. 

\begin{lem}\label{lem:gy-h}
Assume that $g,h\in R[x,y]^{\phi }$ and $c\in R\sm \zs $ 
satisfy $yg-h\in cR[x,y]$, 
i.e., 
$yg\equiv _ch$. 
Then, 
$c^{-1}aFg$ belongs to $\pl (\phi )$. 

\end{lem}
\begin{proof}
We remark that 
$\pl (\phi )=\delta (R[x,y])\cap R[x,y]^{\phi }[c^{-1}]$, 
since $\delta (R[x,y])\subset R[x,y]$ 
and $R[x,y]\cap R[x,y]^{\phi }[c^{-1}]=R[x,y]^{\phi }$. 
By assumption, 
$s:=c^{-1}(yg-h)$ lies in $R[x,y]$. 
Hence, 
$\delta (s)$ is in $\delta (R[x,y])$. 
Since $\delta $ is a linear map over 
$R[x,y]^{\phi }=\ker \delta $, 
and $g,h\in R[x,y]^{\phi }$, 
we have 
$c\delta (s)=\delta (cs)=\delta (yg-h)=\delta (y)g=aFg$. 
Hence, we get 
$\delta (s)=c^{-1}aFg$, 
which belongs to 
$\delta (R[x,y])\cap R[x,y]^{\phi }[c^{-1}]=\pl (\phi )$. 
\end{proof}

In the following discussions, 
we frequently use the following remark ($\dag $):

\nd 
($\dag $) 
If $c\in R\sm \zs $ divides $b$, 
then we have $q\equiv _cy^p$ and $q_1\equiv _c\rho (y)$ 
(cf.~(\ref{eq:Nagata q}), (\ref{eq:q_1 Nagata})). 

\begin{proof}[Proof of {\rm ``b $\Rightarrow $ c"}]
By b and (\ref{eq:delta ideal}), 
we have $\pl (\phi )\subset dFR[x,y]$. 
Hence, 
it suffices to show that $dF\in \pl (\phi )$ 
by Lemma~\ref{lem:pl principal} (i). 
By Remark~\ref{rem:principality of I}~F, 
there exists $\nu (y)\in R[y]$ such that 
$y\equiv _b\nu (\rho (y))$. 
Then, we have 
$y\equiv _b\nu (\rho (y))\equiv _b\nu (q_1)$ by ($\dag $). 
Since $\nu (q_1)$ is in $R[x,y]^{\phi }$, 
this implies that 
$b^{-1}aF=dF$ belongs to $\pl (\phi )$ by Lemma~\ref{lem:gy-h}. 
\end{proof}

Now, we write $b=b_1^{e_1}\cdots b_r^{e_r}$, 
where 
$b_i$ is a prime element of $B$ 
and $e_i\ge 1$, 
and $b_iB\ne b_jB$ if $i\ne j$. 
The following lemma holds regardless of the principality of $I$.

\begin{lem}\label{lem:Nagata pl key}
The following assertions hold for each $i$. 

\nd{\rm (i)} 
For any $s\in R[y]$, 
there exist $h_1,h_2\in R[q_1,q]$ 
such that $sh_2\equiv _{b_i}h_1$ and $h_2\not\equiv _{b_i}0$.

\nd{\rm (ii)} 
There exists $f_i\in R[q_1,q]\sm b_iR[x,y]$ 
such that $b_i^{-e_i}aFf_i$ belongs to $\pl (\phi )$. 
\end{lem}

\begin{proof}
(i) 
Let $i=1$ for simplicity. 
Set $\ol{R}:=R/b_1R$ and $K:=Q(\ol{R})$. 
We denote the image of $h\in R[x,y]$ in $\ol{R}[x,y]$ by $\ol{h}$. 
Then, 
a surjection 
$R[q_1,q]\ni h\mapsto \ol{h}\in \ol{R}[\ol{\rho }(y),y^p]$ 
is defined thanks to ($\dag $). 
Recall that $\gcd (b,\rho (y))=1$ (cf.~Remark~\ref{rem:NGT}). 
Hence, 
$\ol{\rho }(y)$ is nonzero. 
Since $\rho (y)=\sum _{p\nmid i}t_iy^i$, 
it follows that $\ol{\rho }(y)\not\in K(y^p)$. 
Thus, we get 
$K(y)=K(\ol{\rho }(y),y^p)=Q(\ol{R}[\ol{\rho }(y),y^p])$. 
Since $\ol{s}\in K(y)$, 
we can write $\ol{s}=\widehat{h}_1/\widehat{h}_2$, 
where $\widehat{h}_1,\widehat{h}_2\in \ol{R}[\ol{\rho }(y),y^p]$ 
with $\widehat{h}_2\ne 0$. 
For $i=1,2$, 
choose 
$h_i\in R[q_1,q]$ with $\ol{h}_i=\widehat{h}_i$. 
Then, 
we have 
$sh_2\equiv _{b_1}h_1$ and $h_2\not\equiv _{b_1}0$, 
since $\ol{s}\ol{h}_2=\ol{h}_1$ and $\ol{h}_2\ne 0$. 

(ii) 
Let $E$ be the set of $e\ge 0$ 
for which there exist $g_1,g_2\in R[q_1,q]$ such that 
$yg_2\equiv _{b_1^e}g_1$ and $g_2\not\equiv _{b_1}0$. 
For such $g_1$ and $g_2$, 
we have $b_1^{-e}aFg_2\in \pl (\phi )$ 
by Lemma~$\ref{lem:gy-h}$, 
and $g_2\not\in b_1R[x,y]$. 
Hence, 
it suffices to show that 
there exists $e\in E$ with $e\ge e_1$. 
Clearly, $E$ is not empty. 
Suppose that $e:=\max E<e_1$. 
Choose $g_i=\gamma _i(q_1,q)\in R[q_1,q]$ for $i=1,2$ 
with 
$yg_2\equiv _{b_1^e}g_1$ and $g_2\not\equiv _{b_1}0$, 
where $\gamma _i(x,y)\in R[x,y]$. 
Then, 
by ($\dag $), 
$t:=y\gamma _2(\rho (y),y^p)-\gamma _1(\rho (y),y^p)
\equiv _{b_1^l}yg_2-g_1$ 
holds for all $0\le l\le e_1$, 
and hence for $l=e,e+1$ by supposition. 
Since $yg_2\equiv _{b_1^e}g_1$, 
we get 
$t\equiv _{b_1^e}yg_2-g_1\equiv _{b_1^e}0$, 
i.e., $b_1^{-e}t\in R[y]$. 
Thus, by (i), 
there exist $h_1,h_2\in R[q_1,q]$ 
such that $b_1^{-e}th_2\equiv _{b_1}h_1$ and $h_2\not\equiv _{b_1}0$. 
Then, 
we have 
\begin{align*}
(yg_2-g_1)h_2-b_1^eh_1
\equiv _{b_1^{e+1}}th_2-b_1^eh_1 
=b_1^e(b_1^{-e}th_2-h_1) 
\equiv _{b_1^{e+1}}0. 
\end{align*}
This gives that 
$yg_2h_2\equiv _{b_1^{e+1}}g_1h_2+b_1^eh_1$. 
Moreover, 
we have 
$g_2h_2\not\equiv _{b_1}0$, 
since $g_2\not\equiv _{b_1}0$ and $h_2\not\equiv _{b_1}0$. 
This contradicts the maximality of $e$. 
\end{proof}

\begin{proof}[Proof of {\rm ``d $\Rightarrow $ b"}]
$\pl (\phi )$ contains $\delta (y)=aF$ 
and $b^{-e_i}aFf_i$ for each $i$, 
where $f_i$ is as in Lemma~\ref{lem:Nagata pl key} (ii). 
Since $f_i\not\in b_iR[x,y]$, we have 
$\gcd (aF,b^{-e_1}aFf_1,\ldots ,b^{-e_r}aFf_r)
=(b_1^{e_1}\cdots b_r^{e_r})^{-1}aF=dF$, 
which is in $R[x,y]^{\phi }$. 
Hence, 
by d and Lemma~\ref{lem:pl principal} (ii), 
$dF$ lies in $\pl (\phi )$. 
Thus, 
$dF$ is in $FI$ 
by (\ref{eq:delta ideal}), 
and so $d\in I$. 
This implies b. 
\end{proof}

\subsection{Conductor}\label{sect:Dedekind}

Let $p$ be a prime number, 
$S$ a ring with $\ch S=p$, 
and $S[y]$ the polynomial ring in one variable over $S$. 
The purpose of \S \ref{sect:Dedekind} is to prove the following theorem.

\begin{thm}\label{thm:Dedekind}
For every $f\in S[y]$, 
we have 
$(f')^pS[y]\subset S[y^p,f]$. 
\end{thm}

We make use of 
the following well-known fact (cf.~\cite[Theorem 12.1.1]{HS}). 

\begin{lem}\label{lem:conductor}
Let $R$ be an integrally closed domain, 
$L$ an algebraic extension of $Q(R)$, 
$z\in L$, 
and $\ol{R[z]}$ the integral closure of $R[z]$ 
in $Q(R[z])$. 
If $z$ separable over $Q(R)$, 
and the minimal polynomial $\Phi (T)$ of $z$ over $Q(R)$ 
lies in $R[T]$, 
then we have $\Phi '(z)\ol{R[z]}\subset R[z]$. 
Namely, 
$\Phi '(z)$ is in the conductor of $R[z]$. 

\end{lem}

When $f$ is monic, 
it is not difficult to derive Theorem~\ref{thm:Dedekind} 
from Lemma~\ref{lem:conductor}. 
However, 
for the general case, 
an additional argument is needed.

\begin{lem}\label{lem:Dedek1}
Let $S:=\bF _p[\xi _1,\ldots ,\xi _{d-1}]$ be the polynomial ring 
in $d-1$ variables over $\bF _p$ with $p\nmid d$, 
and let $g:=y^d+\sum _{i=1}^{d-1}\xi _iy^i\in S[y]$. 
Then, for each $l\ge 0$, 
there exist $f_1,\ldots ,f_p\in S[y^p]$ 
such that 
{\rm (i)} $(g')^py^l=\sum _{i=0}^{p-1}f_{p-i}g^i$, 
and 
{\rm (ii)} the total degree of 
$f_{p-i}$ in $\xi _1,\ldots ,\xi _{d-1}$ 
is at most $p-i$ for $i=0,\ldots ,p-1$. 
\end{lem}

\begin{proof}
(i) Since $g^p\in S[y^p]$, 
we have $S[y^p,g]=\sum _{i=0}^{p-1}S[y^p]g^i$. 
Hence, 
it suffices to show that $(g')^py^l\in S[y^p,g]$. 
We prove that $(g')^pS[y]\subset S[y^p,g]$ 
using Lemma~\ref{lem:conductor} with $R:=S[g]$, $z:=y^p$ 
and $\Phi (T):=T^d+\sum _{i=1}^{d-1}\xi _i^pT^i-g^p\in S[g][T]$. 
Note that $\Phi (y^p)=0$ and $\Phi '(y^p)=(g')^p$. 
Hence, 
it suffices to 
check the following:

\nd (1) $S[g]$ is an integrally closed domain. 

\nd (2) 
$S[y]\subset Q(S[g,y^p])$, and $S[y]$ is integral over $S[g,y^p]$.

\nd (3) 
$\Phi (T)$ is an irreducible polynomial in $T$ over $Q(S[g])$.

Actually, 
(3) implies that $\Phi (T)$ 
is the minimal polynomial of $y^p$ over $Q(S[g])$. 
Since $p\nmid d$ by assumption, 
it follows that $y^p$ is separable over $Q(S[g])$.

Since $S[g]$ is a polynomial ring 
in $d$ variables over $\bF _p$, 
it is an integrally closed domain, 
hence (1). 
Since $[Q(S[y]):Q(S[y^p])]=p$ and $g\in Q(S[y])\sm Q(S[y^p])$, 
we have $Q(S[y^p,g])=Q(S[y])\supset S[y]$. 
Since $y$ is integral over $S[y^p,g]$, 
we get (2). 
Since $\gcd (d,p)=1$, 
we see that $\Phi (T)$ is an irreducible 
element of $Q(S)[g,T]\simeq Q(S)[x_1,x_2]$. 
Hence, 
the polynomial $\Phi (T)$ in $T$ is irreducible 
over $Q(S)[g]$, 
and thus over $Q(S[g])$ by Gauss's lemma. 
This proves (3).

(ii) 
We define a monomial order on $S[y]=\bF _p[\xi _1,\ldots ,\xi _{d-1},y]$ by 
$\xi _1^{i_1}\cdots \xi _{d-1}^{i_{d-1}}y^{i_d}
\succeq \xi _1^{j_1}\cdots \xi _{d-1}^{j_{d-1}}y^{j_d}$ 
if $\sum _{l=1}^{d-1}i_l>\sum _{l=1}^{d-1}j_l$, 
or $\sum _{l=1}^{d-1}i_l=\sum _{l=1}^{d-1}j_l$ 
and $(i_1,\ldots ,i_d)\ge _{\rm lex}(j_1,\ldots ,j_d)$, 
where $\ge _{\rm lex}$ is the lexicographic order. 
We denote by $\lt (h)$ the leading term of $h\in S[y]$ for this order. 
Then, we have 
$\lt (g)=\xi _1y$ and $\lt (g')=\xi _1$. 
Let $m:=\xi _1^{k_1}\cdots \xi _{d-1}^{k_{d-1}}y^{k_d}$ 
be the maximum among 
$\lt (f_{p-i}g^i)=\lt (f_{p-i})\xi _1^iy^i$ for $0\le i<p$, 
and $J$ the set of $0\le i<p$ with $\lt (f_{p-i}g^i)=m$.

Now, 
suppose that the total degree of $f_{p-i}$ 
in $\xi _1,\ldots ,\xi _{d-1}$ 
is greater than $p-i$ for some $i$. 
Then, 
that of $\lt (f_{p-i})\xi _1^iy^i$ is greater than $p$. 
Hence, 
$\sum _{l=1}^{d-1}k_l>p$ holds by maximality. 
Thus, 
$\lt ((g')^py^l)=\xi _1^py^l$ cannot be $m$. 
Since $(g')^py^l=\sum _{i=0}^{p-1}f_{p-i}g^i$, 
this implies that $|J|\ge 2$, 
for otherwise 
$\lt (\sum _{i=0}^{p-1}f_{p-i}g^i)=m$. 
Choose $i_1,i_2\in J$ with $i_1<i_2$. 
Then, 
we have 
$\lt (f_{p-i_1})\xi _1^{i_1}y^{i_1}=
\xi _1^{k_1}\cdots \xi _{d-1}^{k_{d-1}}y^{k_d}
=\lt (f_{p-i_2})\xi _1^{i_2}y^{i_2}$. 
It follows that 
$\deg _y\lt (f_{p-i_l})=k_d-i_l$ 
for $l=1,2$. 
Since $f_{p-i_l}$ is in $S[y^p]$, 
we have $k_d-i_l\in p\Z $. 
Thus, $i_2-i_1$ is in $p\Z $. 
This contradicts that $0\le i_1<i_2<p$. 
\end{proof}

\begin{proof}[Proof of Theorem~$\ref{thm:Dedekind}$]
Observe that $(f+h)'=f'$ and $S[y^p,f+h]=S[y^p,f]$ 
for any $h\in S[y^p]$. 
Hence, 
replacing $f$ with $f+h$ for some $h\in S[y^p]$, 
we may assume that $f(0)=0$ 
and $d:=\deg f\not\in p\Z $. 
We may also assume that $f\ne 0$. 
Now, write $f=\sum _{i=1}^du_iy^i$, 
where $u_i\in S$. 
To show $(f')^pS[y]\subset S[y^p,f]$, 
it suffices to verify that 
$(f')y^l\in \bF _p[u_1,\ldots ,u_d,y^p,f]$ 
for each $l\ge 0$. 
For this purpose, 
we may assume that $u_1,\ldots ,u_d$ 
are algebraically independent over $\bF _p$. 
Set $\xi _i=u_i/u_d$ for $1\le i<d$ and 
$g:=u_d^{-1}f=y^d+\sum _{i=1}^{d-1}\xi _iy^i$. 
Choose $f_1,\ldots ,f_p\in \bF _p[\xi _1,\ldots ,\xi _{d-1},y^p]$ 
as in Lemma~\ref{lem:Dedek1}. 
Then, 
since $f=u_dg$ and $f'=u_dg'$, 
we have 
$(f')^py^l=u_d^p(g')^py^l
=u_d^{p}\sum _{i=0}^{p-1}f_{p-i}g^i
=\sum _{i=0}^{p-1}u_d^{p-i}f_{p-i}f^i$ by (i). 
By (ii), 
$u_d^{p-i}f_{p-i}$ is in $\bF _p[u_1,\ldots ,u_d,y^p]$ for each $i$. 
Therefore, 
$(f')y^l$ belongs to $\bF _p[u_1,\ldots ,u_d,y^p,f]$. 
\end{proof}

\subsection{Invariant ring: generators}
\label{sect:Nagata invariant ring1}

In \S \ref{sect:Nagata invariant ring1}, 
we determine the generators of $R[x,y]^{\phi }$. 
The main task is to construct a new element $q_2\in R[x,y]^{\phi }$. 
For this purpose, 
we need Theorem~\ref{thm:Dedekind}.

Throughout \S \ref{sect:Nagata invariant ring1}, 
let $\ol{R}:=R/bR$, 
and $\ol{h}$ denotes the image of $h\in R[x,y]$ in $\ol{R}[x,y]$. 
Recall that $\ol{q}=y^p$ and $\ol{q}_1=\ol{\rho }(y)$ by ($\dag $) 
after Lemma~\ref{lem:gy-h}.

\begin{rem}\label{rem:R[bx,y]}\rm 
(i) 
For $h\in R[x,y]$, 
we have $\deg _x\ol{h}\le 0$ if $h\in R[bx,y]$, 
and $\deg _x\ol{h}\le 1$ 
if $h\in b^{-1}R[bx,y]\cap R[x,y]$.

\nd (ii) 
Since $f \in R[bx,y]$, 
we have $F,q,q_1\in R[bx,y]$ 
(cf.~(\ref{eq:Nagata q}) and (\ref{eq:q_1 Nagata})). 
\end{rem}

Set $\xi (T):=q_1-\rho (T)\in R[q_1][T]$. 
Then, 
by (\ref{eq:q_1 Nagata}), 
we have 
$\xi (y)=bw$ for some 
\begin{equation}\label{eq:w Nagata}
w\in x+(db)^{p-2}F^{p-1}yR[f,aF,y]
\subset x+R[bx,y]. 
\end{equation}
Since $\phi (y)=y+bdF$, 
we can use Remark~\ref{rem:q_1} with 
$(a,b,c,S)=(b,dF,y,R[q_1])$. 
Set $\xi ^p(T):=q_1^p-\rho ^p(T)$, 
where $\rho ^p(T):=\sum _{p\nmid i}t_i^pT^i$. 
Then, 
by (\ref{eq:q_1}), 
$R[x,y]^{\phi }$ contains 
\begin{equation}\label{eq:tilde q_1 Nagata}
\widetilde{q}_1:=b^{1-p}(q_1^p-\rho ^p(q))
=b^{1-p}\xi ^p(q)
\in 
bw^p+\rho '(y)^p{\cdot }(dF)^{p-1}y+b^{p-1}R[bx,y], 
\end{equation}
where we use $F,q_1\in R[bx,y]$. 
By Theorem~\ref{thm:Dedekind}, 
there exists $\lambda (x,y)\in R[x,y]$ 
such that $\rho '(y)^py=\lambda (y^p,\rho (y))$. 
Then, 
$v:=(\rho '(y)^py-\lambda (q,q_1))(dF)^{p-1}$ belongs to $bR[x,y]$ 
by ($\dag $) after Lemma~\ref{lem:gy-h}, 
and to 
$R[bx,y]$ 
by Remark~\ref{rem:R[bx,y]} (ii). 
Hence, 
$b^{-1}v$ belongs to $R[x,y]\cap b^{-1}R[bx,y]$. 
Thus, 
by (\ref{eq:tilde q_1 Nagata}) and (\ref{eq:w Nagata}), 
we see that 
\begin{equation}\label{eq:q_2 Nagata}
\begin{aligned}
q_2:=b^{-1}(\widetilde{q}_1-\lambda (q,q_1)(dF)^{p-1})
&\in w^p+b^{-1}v+b^{p-2}R[bx,y] \\
&\subset x^p+R[bx,y]+R[x,y]\cap b^{-1}R[bx,y]. 
\end{aligned}
\end{equation}
Since 
$\widetilde{q}_1$ and $\lambda (q,q_1)(dF)^{p-1}$ are in $R[x,y]^{\phi }$, 
so is $q_2$. 
Therefore, we can define 
\begin{equation}\label{eq:Nagata sigma}
\sigma :\Ry \ni h(y_0,y_1,y_2)\mapsto h(q,q_1,q_2)\in R[x,y]^{\phi },\ 
\text{ where }\ 
\y :=\{ y_0,y_1,y_2\} .  
\end{equation}

\begin{example}
If $p\ge 3$ in Example~\ref{example:Nagata}, 
then $d=\gcd (z,2y)=1$, 
$b=a=z$, 
$\theta (y)=\rho (y)=\rho ^p(y)=y^2$ and $\theta ^*(y)=0$. 
Hence, we have 
$q_1=f$ and $\widetilde{q}_1=z^{1-p}(f^p-q^2)$. 
Since $\rho '(y)^py=2y^{p+1}$, 
we may take $\lambda (x,y)=2y^{(p+1)/2}$. 
Then, 
we get 
$q_2=z^{-1}\bigl(\widetilde{q}_1-2f^{(p+1)/2}f^{p-1}\bigr) $. 

\end{example}

\begin{rem}\label{rem:q_1 etc}\rm 
(i) We have $f\in R[q,q_1]$ by (\ref{eq:q_1 Nagata}), 
and so $\widetilde{q}_1\in R[q,q_1,q_2]$ by (\ref{eq:q_2 Nagata}).

\nd (ii) We have 
$\widetilde{q}_1,q_2\in R_b[q,q_1]$ 
by (\ref{eq:tilde q_1 Nagata}) and (\ref{eq:q_2 Nagata}), 
since $f\in R[q,q_1]$.

\nd (iii) $\ol{q}_2$ is a monic polynomial in $x$ of degree $p$ 
by (\ref{eq:q_2 Nagata}) and Remark~\ref{rem:R[bx,y]} (i).

\nd (iv) 
$q$ and $q_1$ are algebraically independent over $R_a$, 
because $R_a[x,y]$ is algebraic over 
$R_a[x,y]^{\widetilde{\phi }}=R_a[f,q]=R_a[q_1,q]$ 
(cf.~\S \ref{sect:Nagata construction}). 

\end{rem}

The following is the main result of \S \ref{sect:Nagata invariant ring1}. 

\begin{thm}\label{thm:Nagata2}
In the notation above, 
we have $R[x,y]^{\phi }=R[q,q_1,q_2]$. 
\end{thm}
\begin{proof}
Since $q_2\in R_b[q,q_1]$ by Remark~\ref{rem:q_1 etc} (ii), 
we have $R_b[q,q_1]=R_b[q,q_1,q_2]$. 
Hence, 
by Theorem~\ref{thm:Nagata1}, 
it suffices to show that 
$R_b[q,q_1,q_2]\cap R[x,y]=R[q,q_1,q_2]$. 
Let $\widehat{\sigma }:\ol{R}[\y ]\to \ol{R}[x,y]$ 
be the substitution map induced by $\sigma $ in (\ref{eq:Nagata sigma}). 
We show that 
$\ker \widehat{\sigma }=(y_1^p-\ol{\rho ^p}(y_0))$. 
Then, the assertion follows by Lemma~\ref{lem:intersection 2}, 
since $\sigma (y_1^p-\rho ^p(y_0))
=q_1^p-\rho ^p(q)=b^{p-1}\widetilde{q}_1$ 
by (\ref{eq:tilde q_1 Nagata}), 
and $\widetilde{q}_1\in R[q,q_1,q_2]$ by Remark~\ref{rem:q_1 etc}~(i).

Since $\ol{q}=y^p$ and $\ol{q}_1=\ol{\rho }(y)$ by ($\dag $), 
we have $\widehat{\sigma }(y_1^p-\ol{\rho ^p}(y_0))=
\ol{\rho }(y)^p-\ol{\rho ^p}(y^p)=0$. 
To show 
$\ker \widehat{\sigma }\subset (y_1^p-\ol{\rho ^p}(y_0))$, 
pick any $\eta =\sum _i\eta _iy_2^i\in \ker \widehat{\sigma }$, 
where $\eta _i\in \ol{R}[y_0,y_1]$. 
We claim that 
$\widehat{\sigma }(\eta _i)=\eta _i(y^p,\ol{\rho }(y))=0$ 
for all $i$. 
In fact, 
if $\{ i\mid \widehat{\sigma }(\eta _i)\ne 0\} \ne \emptyset $, 
and $j:=\max \{ i\mid \widehat{\sigma }(\eta _i)\ne 0\} $, 
then 
$\widehat{\sigma }(\eta )=
\sum _i\widehat{\sigma }(\eta _i)\ol{q}_2^i
=\eta _j(y^p,\ol{\rho }(y))x^{jp}+\cdots \ne 0$ 
by Remark~\ref{rem:q_1 etc} (iii), 
a contradiction. 
Hence, 
we may assume that $\eta \in \ol{R}[y_0,y_1]$. 
Let $\nu \in \ol{R}[y_0,y_1]$ be the remainder of $\eta $ 
divided by $y_1^p-\ol{\rho ^p}(y_0)$ 
as a polynomial in $y_1$. 
Then, $\nu $ lies in $\ker \widehat{\sigma }$, 
since $\eta ,y_1^p-\ol{\rho ^p}(y_0)\in \ker \widehat{\sigma }$. 
By Lemma~\ref{lem:remainder Nagata} below, 
this implies that $\nu =0$, 
i.e., 
$\eta \in (y_1^p-\ol{\rho ^p}(y_0))$. 
\end{proof}

\begin{lem}\label{lem:remainder Nagata}
$\nu (y^p,\ol{\rho }(y))\ne 0$ holds for every 
$\nu \in \ol{R}[y_0,y_1]\sm \zs $ with $\deg _{y_1}\nu <p$. 
\end{lem}
\begin{proof}
Suppose that 
there exists $\nu \in \ol{R}[y_0,y_1]\sm \zs $ 
with $n:=\deg _{y_1}\nu <p$ and $\nu (y^p,\ol{\rho }(y))=0$. 
Then, 
$\nu $ is not in $\ol{R}[y_0]$, 
i.e., 
$n\ge 1$. 
Hence, 
$\nu _{y_1}:=\partial \nu /\partial y_1$ is 
of $y_1$-degree $n-1\ge 0$. 
Choose $\nu $ with least $n$. 
By the chain rule, 
we have 
$$ 
0=(\nu (y^p,\ol{\rho }(y)))'
=py^{p-1}\nu _{y_0}(y^p,\ol{\rho }(y))+
\ol{\rho }'(y)\nu _{y_1}(y^p,\ol{\rho }(y))
=\ol{\rho }'(y)\nu _{y_1}(y^p,\ol{\rho }(y)).
$$
Recall that $\gcd (b,\rho '(y))=1$ (cf.~Remark~\ref{rem:NGT}). 
Hence, 
$\ol{\rho }'(y)$ is not a zero-divisor of $\ol{R}[y]$. 
Thus, we get 
$\nu _{y_1}(y^p,\ol{\rho }(y))=0$. 
This contradicts the minimality of $n$. 
\end{proof}

\subsection{Invariant ring: relation}
\label{sect:Nagata invariant ring2}

By Theorem~\ref{thm:Nagata2}, 
$\sigma $ in (\ref{eq:Nagata sigma}) is surjective. 
Since $f$ is in $R[q,q_1]$ by Remark~\ref{rem:q_1 etc}~(i), 
we can write $(dF)^{p-1}=\nu (q,q_1)$, 
where $\nu \in R[y_0,y_1]$. 
We define 
\begin{equation}\label{eq:Nagata relation}
\Lambda :=b^py_2+\rho ^p(y_0)-y_1^p+b^{p-1}\lambda (y_0,y_1)\nu (y_0,y_1). 
\end{equation}
Then, 
from (\ref{eq:tilde q_1 Nagata}) and (\ref{eq:q_2 Nagata}), 
we see that $\sigma (\Lambda )=0$. 
Hence, 
$\ker \sigma $ contains $(\Lambda )$.

\begin{thm}\label{thm:Nagata relation}
We have $\ker \sigma =(\Lambda )$. 
Hence, 
$R[x,y]^{\phi }$ is isomorphic to $\Ry /(\Lambda )$ 
as an $R$-algebra, 
where $\y =\{ y_0,y_1,y_2\} $.  
\end{thm}
\begin{proof}
Pick any $h\in \ker \sigma $. 
Noting $b^p\in R_b^*$, 
we can write $h=\Lambda h_1+h_0$, 
where $h_1\in R_b[\y ]$ and $h_0\in R_b[y_0,y_1]$. 
Then, 
we have $h_0(q,q_1)=0$, 
since $h,\Lambda \in \ker \sigma $. 
By Remark~\ref{rem:q_1 etc} (iv), 
this implies that $h_0=0$, i.e., $h=\Lambda h_1$. 
Since $h,\Lambda \in \Ry $, $h_1\in R_b[\y ]$ and 
$\gcd (\Lambda ,b)=\gcd (\rho ^p(y_0)-y_1^p,b)=1$, 
it follows that $h_1\in \Ry $. 
This proves $h\in (\Lambda )$. 
\end{proof}

In the rest of \S \ref{sect:Nagata invariant ring2}, 
we study the structure of 
$R[x,y]^{\phi }\simeq \Ry /(\Lambda )$.

\begin{thm}\label{thm:Nagata relation1}
If $I$ is a principal ideal of $R[x,y]$, 
then $\Lambda $ is a coordinate of $R[y_1][y_0,y_2]$, 
i.e., $\Ry =R[y_1,\Lambda ,L]$ 
for some $L\in \Ry $. 
Hence, 
$R[x,y]^{\phi }$ is isomorphic to $R[x,y]$ 
as an $R$-algebra. 
\end{thm}
\begin{proof}
Set $S:=R[y_1]$. 
Write $\Lambda =b^py_2+\sum _{i\ge 0}u_iy_0^i$, 
where $u_i\in S$. 
Due to a well-known result of Russell~\cite{Russell} 
and Sathaye~\cite{Sathaye}, 
it suffices to verify that 
$\ol{u}_1\in (S/b^pS)^*$ 
and $\ol{u}_i\in \nil (S/b^pS)$ if $i\ge 2$. 
Here, $\ol{u}$ denotes the image of $u$ in $S/b^pS$ for $u\in S$.

Since $\rho ^p(y_0)=\sum _{p\nmid i}t_i^py_0^i$, 
we see from (\ref{eq:Nagata relation}) that 
$u_i\in bS$ if $i>0$ and $p\mid i$, 
and $u_i\in t_i^p+bS$ if $p\nmid i$. 
Hence, 
$\ol{u}_i$ is in $\nil (S/b^pS)$ if $i>0$ and $p\mid i$. 
If $p\nmid i$, 
then $\ol{u}_i$ is in $\ol{t}_i^p+\nil (S/b^pS)$. 
Since $I$ is principal by assumption, 
we know by Remark~\ref{rem:principality of I}~E that 
$t_1r_1+br_2=1$ for some $r_1,r_2\in R$, 
and $t_i\in \sqrt{bR}$ if $p\nmid i$ and $i\ge 2$. 
Since $t_1^pr_1^p+b^pr_2^p=1$, 
we have $\ol{t}_1^p\in (S/b^pS)^*$. 
Hence, 
we get 
$\ol{u}_1\in \ol{t}_1^p+\nil (S/b^pS)\subset (S/b^pS)^*$. 
When $t_i\in \sqrt{bR}$, 
we have $\ol{u}_i\in \ol{t}_i^p+\nil (S/b^pS)\subset \nil (S/b^pS)$. 
\end{proof}

To prove the non-polynomiality of $\Ry /(\Lambda )$, 
we use the following lemma. 

\begin{lem}\label{lem:non-sing}
Let $S$ be a domain, 
$\kappa $ an algebraic closure of $Q(S)$, 
and $h\in \Sy $, 
where $\y :=\{ y_0,\ldots ,y_r\} $ 
is a set of variables. 
If $\Sy /(h)\simeq S[y_1,\ldots ,y_r]$ 
and $h$ is irreducible in $\kappa [\y ]$, 
then we have 
$\kappa [\y ]/(h)\simeq \kappa [y_1,\ldots ,y_r]$. 
Hence, the hypersurface $h=0$ in $\A _{\kappa }^{r+1}$ 
is isomorphic to $\A _{\kappa }^r$, 
and thus smooth. 
Consequently, 
the system of equations $h=0$ and $\partial h/\partial y_i=0$ 
for $i=0,\ldots ,r$ has no solution in $\kappa ^{r+1}$. 
\end{lem}
\begin{proof}
By assumption, 
the $S$-algebra $\Sy /(h)$ is generated by $r$ elements. 
Hence, 
there exist $\xi _1,\ldots ,\xi _r\in \Sy $ 
such that $y_0,\ldots ,y_r\in S[\xi _1,\ldots ,\xi _r]+h\Sy $. 
Since 
$S[\xi _1,\ldots ,\xi _r]+h\Sy \subset 
\kappa [\xi _1,\ldots ,\xi _r]+h\kappa [\y ]$, 
we see that the $\kappa $-algebra 
$\kappa [\y ]/(h)$ is generated by the images 
$\ol{\xi }_1,\ldots ,\ol{\xi }_r$ 
of $\xi _1,\ldots ,\xi _r$. 
Since $h$ is irreducible in $\kappa [\y ]$ by assumption, 
$\kappa [\y ]/(h)$ is a $\kappa $-domain 
with $\trd _{\kappa }\kappa [\y ]/(h)=r$. 
Therefore, 
$\ol{\xi }_1,\ldots ,\ol{\xi }_r$ 
must be algebraically independent over $\kappa $. 
The last part is well known. 
\end{proof}

\begin{thm}\label{thm:Nagata relation2}
Assume that $I$ is not a principal ideal of $R[x,y]$. 

\nd{\rm (i)} 
$R[x,y]^{\phi }$ is not isomorphic to 
$R[x,y]$ as an $R$-algebra. 

\nd{\rm (ii)} 
If $R=k[z_1,\ldots ,z_n]$ 
is the polynomial ring in $n$ 
variables over a field $k$ with $\ch k>0$, 
where $n\ge 1$, 
then $R[x,y]^{\phi }$ is not isomorphic to 
$R[x,y]$ as a $k$-algebra. 
\end{thm}
\begin{proof}
(i) By Remark~\ref{rem:principality of I} E, 
we have 
$1\not\in (t_1,b)$, 
or $t_i\not\in \sqrt{bR}$ for some $i\ge 2$ with $p\nmid i$. 
There exists $\fp \in \Spec R$ such that $t_1,b\in \fp $ in the former case, 
and $t_{i}\not\in \fp $ and $b\in \fp $ in the latter case. 
In both cases, 
$\ol{\rho }'(y_0)=\sum _{p\nmid i}i\ol{t}_iy_0^{i-1}$ 
is not a nonzero constant, 
and $\ol{\Lambda }=\ol{\rho ^p}(y_0)-y_1^p$. 
Here, 
$\ol{h}$ denotes the image of $h\in R[\y ]$ in 
$(R/\fp )[\y ]$. 
Let $\kappa $ be an algebraic closure of $Q(R/\fp )$. 
Then, there exists $\alpha \in \kappa $ with $\ol{\rho }'(\alpha )=0$.

Now, 
suppose that 
$R[x,y]^{\phi }\simeq \Ry /(\Lambda )=:A$ is isomorphic to $R[x,y]$ 
as an $R$-algebra. 
Then, we have 
$A/\fp A\simeq R[x,y]/\fp R[x,y]\simeq (R/\fp )[x,y]$. 
Since 
$A/\fp A\simeq \Ry /(\Lambda \Ry +\fp \Ry )\simeq 
(R/\fp )[\y ]/(\ol{\Lambda })$, 
we get 
$(R/\fp )[\y ]/(\ol{\Lambda })\simeq (R/\fp )[x,y]$. 
This implies $\ol{\Lambda }\ne -y_1^p$, 
that is, 
$\ol{\rho ^p}(y_0)\ne 0$. 
Then, 
$\ol{\Lambda }=\ol{\rho ^p}(y_0)-y_1^p$ is irreducible in $\kappa [\y ]$, 
since the degree of $\ol{\rho ^p}(y_0)=\sum _{p\nmid i}\ol{t_i^p}y_0^i$ 
is coprime to $p$. 
Thus, 
the assumption of Lemma~\ref{lem:non-sing} holds 
for $S=R/\fp $ and $h=\ol{\Lambda }$. 
However, 
$(y_0,y_1,y_2)=(\alpha ^p,\ol{\rho }(\alpha ),0)$ is a solution of 
$\ol{\Lambda }=\partial \ol{\Lambda }/\partial y_i=0$ 
for $i=0,1,2$, 
since $\ol{\rho ^p}(\alpha ^p)=\ol{\rho }(\alpha )^p$ 
and $(\ol{\rho ^p})'(\alpha ^p)=\ol{\rho }'(\alpha )^p=0$. 
This is a contradiction.

(ii) 
Let $\kappa $ be an algebraic closure of $k$, 
and $\z :=\{ z_1,\ldots ,z_n\} $. 
Then, 
$\Lambda $ is irreducible in $\kappa [\y ,\z ]$, 
since $\Lambda $ is a linear, primitive polynomial in $y_2$ 
over $\kappa [y_0,y_1,\z ]$. 
Now, 
suppose that 
$R[x,y]^{\phi }\simeq k[\y ,\z ]/(\Lambda )$ 
is isomorphic to $R[x,y]=k[x,y,\z ]$ 
as a $k$-algebra. 
Then, 
the assumption of Lemma~\ref{lem:non-sing} 
holds for $S=k$ and $h=\Lambda $.

Since $R=k[\z ]$, 
we write $\lambda (y_0,y_1)=\lambda (y_0,y_1,\z )$ 
and $\rho (y_0)=\rho (y_0,\z )$. 
As in (i), 
we have $1\not\in (t_1,b)$, 
or $t_i\not\in \sqrt{bR}$ for some $i\ge 2$ with $p\nmid i$. 
By Hilbert's Nullstellensatz, 
there exists $\gamma \in \kappa ^n$ such that 
$t_1(\gamma )=b(\gamma )=0$ in the former case, 
and $t_i(\gamma )\ne 0$ and $b(\gamma )=0$ in the latter case. 
In both cases, 
$\rho _{y_0}(y_0,\gamma )$ 
is not a nonzero constant, 
and $b(\gamma )=0$. 
Choose $\alpha \in \kappa $ with $\rho _{y_0}(\alpha ,\gamma )=0$. 
Then, 
we have 
$\lambda (\alpha ^p,\rho (\alpha ,\gamma ),\gamma )
=\rho _{y_0}(\alpha ,\gamma )^p\alpha =0$, 
since $\lambda (y^p,\rho (y))=\rho '(y)^py$ 
by the choice of $\lambda $ 
(cf.~\S \ref{sect:Nagata invariant ring1}). 
Noting this and 
$\partial b^p/\partial z_i=\partial \rho ^p(y_0)/\partial z_i=0$ for all $i$, 
we can check that 
$(y_0,y_1,y_2,\z )=(\alpha ^p,\rho (\alpha ,\gamma ),0,\gamma )$ 
is a solution of 
$\Lambda =\partial \Lambda /\partial y_i
=\partial \Lambda /\partial z_j=0$ 
for $i=0,1,2$ and $j=1,\ldots ,n$. 
This contradicts Lemma~\ref{lem:non-sing}. 
\end{proof}

\section{Question and Conjecture}\label{sect:remark}
\setcounter{equation}{0}

\nd (1) 
The {\it Stable Tameness Conjecture} asserts that 
every $\phi \in \Aut _k\kx $ is {\it stably tame}, 
i.e., 
there exists $l>0$ such that $\phi _l\in \T _{n+l}(k)$, 
where $\phi _l$ is the extension of $\phi $ 
defined by $\phi _l(x_i)=x_i$ for all $i>n$ 
(cf.~\cite[Conjecture 6.1.8]{Essen}). 
When $n=3$, 
every element of 
$T:=\langle \T _3(k)\cup \Aut _{k[x_3]}\kx \rangle $ is 
stably tame 
due to Berson-van den Essen-Wright~\cite{BEW}. 
However, 
we do not know the answer to the following question.

\begin{q}\label{q:STC}\rm
Is $\phi $ in (\ref{eq:simple example}) 
(or more generally $\ep _h$ in \S \ref{sect:rank3 exp}) 
stably tame?
\end{q}

\nd (2) 
In the case $p=0$, 
the author \cite{wild3} studied in detail 
when $\phi \in \Ex _3(k)$ 
belongs to $\T _3(k)$ or not, 
using the Shestakov-Umirbaev theory~\cite{SU} 
and its generalization~\cite{tame3}. 
There, 
he arrived at the following conjecture for $p=0$ 
(cf.~\cite[Conjecture 17.3]{Sugaku}). 
Here, 
we say that $\tau \in  \Aut _k\kx $ is {\it triangular} 
if $\tau (x_i)\in k[x_1,\ldots ,x_i]$ for each $i$.

\begin{conj}\label{conj:tame3}\rm 
For every $\phi \in \Ex _3(k)\cap \T _3(k)$, 
there exists $\sigma \in \T _3(k)$ 
such that 
$\sigma \circ \phi \circ \sigma ^{-1}$ is triangular. 
\end{conj}

It seems reasonable to expect that Conjecture~\ref{conj:tame3} 
also holds for $p>0$. 
In fact, we have the following conjecture 
(see also Question~\ref{q:C=E?} below).

\begin{conj}\label{conj:ch triangular}\rm 
Assume that $p>0$. Then, 
for every $\phi \in \Ch _3(k)\cap \T _3(k)$, 
there exists $\sigma \in \T _3(k)$ such that 
$\sigma \circ \phi \circ \sigma ^{-1}$ is triangular. 
\end{conj}

We claim that Conjecture~\ref{conj:ch triangular} 
implies the following conjecture.

\begin{conj}\label{conj:variable}\rm 
Assume that $p>0$. Then, 
for every $\phi \in \Ch _3(k)\cap \T _3(k)$, 
there exists $\sigma \in \T _3(k)$ such that 
$\sigma (x_1)\in \kx ^{\phi }$. 
Hence, we have $\gamma (\kx ^{\phi })\ge 1$. 
\end{conj}

In fact, 
if $\tau \in  \Ch _3(k)$ is triangular, 
then $\tau $ restricts to an element of $\Ch _2(k)$, 
to which we can apply Theorem~\ref{thm:Osaka}.

\nd (3) 
We note that the Laurent polynomial ring 
$B=k[x_1^{\pm 1},\ldots ,x_p^{\pm 1}]$ 
admits no non-identity 
exponential automorphism. 
In fact, 
every $\Ga $-action on $B$ fixes 
$x_i$ and $x_i^{-1}$ for all $i$, 
since $x_ix_i^{-1}=1\in B^{\Ga }$, 
and $B^{\Ga }$ is factorially closed in $B$ 
(cf.~Remark~\ref{rem:Miyanishi}). 
Clearly, 
the automorphism of $B$ defined by 
$x_1\mapsto x_2\mapsto \cdots \mapsto x_p\mapsto x_1$ 
is of order $p$. 
Hence, 
$B$ admits a non-exponential automorphism of order $p$. 
However, 
there exists no such automorphism of $k[x_1,x_2]$ 
because of Theorem~\ref{thm:Osaka}.

\begin{q}\label{q:C=E?}\rm
Assume that $p>0$. 
Does $\Ch _n(k)=\Ex _n(k)$ hold for $n\ge 3$? 
\end{q}

\nd (4) 
Theorems~\ref{thm:Nagata Main}, 
\ref{thm:plinth rank 3}, \ref{thm:rank3 invariant ring} 
and \ref{thm:rank3 invariant ring2} 
support the following conjecture.

\begin{conj}\label{q:pl k[x]}\rm
Assume that $p>0$ and $n=3$. 
Then, 
for $\phi \in \Ex _3(k)$, 
we have 
$\kx ^{\phi }\simeq \kx $ if and only if $\pl (\phi )$ 
is a principal ideal of $\kx ^{\phi }$. 
\end{conj}


\begin{thebibliography}{AMS}

\bibitem{AM}
M. F. Atiyah\ and\ I. G. Macdonald, 
{\it Introduction to commutative algebra}, 
Addison-Wesley Publishing Co., Reading, MA, 1969. 



\bibitem{BEW}
J. Berson, A. van den Essen\ and\ D. Wright, 
Stable tameness of two-dimensional polynomial automorphisms 
over a regular ring, Adv. Math. {\bf 230} (2012), no.~4-6, 2176--2197.


\bibitem{MIT}
H. E. A. E. Campbell\ and\ D. L. Wehlau, 
{\it Modular invariant theory}, 
Encyclopaedia of Mathematical Sciences, 139, 
Springer-Verlag, Berlin, 2011. 




\bibitem{Essen}
A. van den Essen, 
{\it Polynomial automorphisms and the Jacobian conjecture}, 
Progress in Mathematics, 190, Birkh\"auser Verlag, Basel, 
2000. 


\bibitem{Frank}G.~Freudenburg, 
Actions of ${\bf G}_a$ on ${\bf A}^3$ 
defined by homogeneous derivations, 
J.\ Pure Appl.\ Algebra {\bf 126} (1998), 169--181. 


\bibitem{Plinth}
G. Freudenburg, {\it Algebraic theory of locally nilpotent derivations}, second edition, Encyclopaedia of Mathematical Sciences, 136, 
Springer-Verlag, Berlin, 2017. 

\bibitem{HS}
C. Huneke\ and\ I. Swanson, 
{\it Integral closure of ideals, rings, and modules}, 
London Mathematical Society Lecture Note Series, 336, 
Cambridge University Press, Cambridge, 2006.


\bibitem{Jung}H.~Jung, 
\textit{\"Uber ganze birationale Transformationen der Ebene}, 
J.\ Reine Angew.\ Math.\ {\bf 184} (1942), 161--174. 



\bibitem{Kulk}W.~van der Kulk, 
On polynomial rings in two variables, 
Nieuw Arch.\ Wisk. (3) {\bf 1} (1953), 33--41. 



\bibitem{tame3}
S. Kuroda, 
Shestakov-Umirbaev reductions 
and Nagata's conjecture on a polynomial 
automorphism, Tohoku Math.\ J.\ {\bf 62} (2010), 75--115.



\bibitem{wild3}
S. Kuroda, Wildness of polynomial automorphisms: applications of the Shestakov-Umirbaev theory and its generalization, in {\it Higher dimensional algebraic geometry}, 103--120, RIMS K\^{o}ky\^{u}roku Bessatsu, B24, 
Res. Inst. Math. Sci. (RIMS), Kyoto.


\bibitem{Sugaku}
S. Kuroda, 
Recent developments in polynomial automorphisms: the solution of Nagata's conjecture and afterwards, 
Sugaku Expositions {\bf 29} (2016), 
177--201.


\bibitem{Maubach}
S. Maubach, 
Invariants and conjugacy classes of triangular polynomial maps, 
J. Pure Appl. Algebra {\bf 219} (2015), no.~12, 5206--5224.


\bibitem{MI}
M. Miyanishi\ and\ H. Ito, 
{\it Algebraic surfaces in positive characteristics---purely 
inseparable phenomena in curves and surfaces}, 
World Scientific Publishing Co. Pte. Ltd., 
Hackensack, NJ, 2021. 


\bibitem{M1}M.~Miyanishi, 
{\it Curves on rational and unirational surfaces}, 
Tata Institute of Fundamental Research Lectures 
on Mathematics and Physics, 60, 
Tata Inst. Fund. Res., Bombay, 1978. 


\bibitem{MiyanishiNagoya}
M.~Miyanishi, 
$G\sb{a}$-action of the affine plane, 
Nagoya Math. J. {\bf 41} (1971), 97--100. 


\bibitem{Miyanishi3}
M. Miyanishi, 
Wild $\Bbb{Z}/p\Bbb{Z}$-actions on algebraic surfaces, 
J. Algebra {\bf 477} (2017), 360--389. 



\bibitem{Nagata}M.~Nagata, 
On Automorphism Group of $k[x,y]$, 
Lectures in Mathematics, Department of Mathematics, 
Kyoto University, Vol.\ 5, 
Kinokuniya Book-Store Co.\ Ltd., Tokyo, 1972. 


\bibitem{Rentschler}
R. Rentschler, 
Op\'erations du groupe additif sur le plan affine, 
C. R. Acad. Sci. Paris S\'er. A-B {\bf 267} (1968), 
384--387. 


\bibitem{Russell}
P. Russell, Simple birational extensions of two dimensional affine rational domains, Compositio Math. {\bf 33} (1976), no.~2, 197--208. 

\bibitem{Sathaye}
A. Sathaye, 
On linear planes, Proc. Amer. Math. Soc. {\bf 56} (1976), 1--7.





\bibitem{SU}
I.~Shestakov and U.~Umirbaev, 
The tame and the wild automorphisms of polynomial rings in three variables, 
J.\ Amer.\ Math.\ Soc.\ {\bf 17} (2004), 197--227. 

\bibitem{Takeda}
Y. Takeda, Artin-Schreier coverings of algebraic surfaces, 
J. Math. Soc. Japan {\bf 41} (1989), no.~3, 415--435. 


\bibitem{Tani}
R. Tanimoto, Pseudo-derivations and modular invariant theory, 
Transform. Groups {\bf 23} (2018), no.~1, 271--297.


\end{thebibliography}
\end{document}